\documentclass[11pt]{amsart}
\input{amssym.def}
\input{amssym}
\newtheorem{pro}{Proposition}[section]
\newtheorem{lem}[pro]{Lemma}

\newtheorem{exa}[pro]{Example}
\newtheorem{theo}[pro]{Theorem}
\newtheorem{defi}[pro]{Definition}
\newtheorem{cor}[pro]{Corollary}
\newtheorem{remk}[pro]{Remark}

\newcommand{\ep}{\varepsilon}
\newcommand{\al}{\alpha}
\newcommand{\om}{\omega}
\newcommand{\vp}{\varphi}
\newcommand{\la}{\lambda}

\newcommand{\sun}{\odot}

\newcommand{\lra}{\longrightarrow}
\newcommand{\lmt}{\longmapsto}

\newcommand{\nrm}[1]{\mbox{ $ \displaystyle \left\| {#1} \right\| $} }

\newcommand{\fk}[1]{ \left( {#1} \right) }

\newcommand{\bk}[1]{ \left\{ {#1} \right\} }
\newcommand{\btr}[1]{\mbox{ $ \left| {#1} \right| $ }}

\newcommand{\ce}{{\bf\Bbb C}}

\newcommand{\scT}{{\mathcal{T}}}

\newcommand{\scTco}{{\mathcal{T}_{co}}}

\newcommand{\re}{{\bf\Bbb R}}
\newcommand{\rep}{{\bf\Bbb R^+}}
\newcommand{\za}{{\bf\Bbb N}}

\newcommand{\jz}{{\bf\Bbb J}}
\newcommand{\scS}{{\mathcal{S}}}
\newcommand{\scA}{{\mathcal{A}}}

\newcommand{\scR}{{\mathcal{R}}}
\newcommand{\scr}{{\mathcal{Z}}}
\newcommand{\scU}{{\mathcal{U}}}
\newcommand{\scV}{{\mathcal{V}}}
\newcommand{\scW}{{\mathcal{W}}}

\newcommand{\Xsr}{X^{\sun}_{rev}}

\newcommand{\fxx}{f_{x,x^{\sun}}}

\newcommand{\seq}[2]{\mbox{$ \bk{ {#1}_{#2} }_{{#2} \in \za} $} }
\newcommand{\net}[3]{\mbox{$ \bk{ {#1}_{#2} }_{{#2} \in {#3}} $} }

\newcommand{\supp}[1]{ \mbox{ $ {\rm supp} \left\{ {#1} \right\} $ } }
\newcommand{\lb}{ \mbox{$ {\rm lb}  $} }

\newcommand{\ilm}[1]{  \lim_{ {#1} \to \infty}  }
\newcommand{\netlim}[2]{  \lim_{ {#1}\in {#2}}  }

\newcommand{\Funk}[5]{ \begin{array}{ccccc}
                       {#1} & : & {#2} & \lra & {#3} \\
                            &   & {#4} & \lmt & \displaystyle{#5} 
                       \end{array}                       }
\newcommand{\funk}[3]{ \begin{array}{ccccc}
                       {#1} & : & {#2} & \lra & {#3}
                       \end{array}                       }

\begin{document}

\title{On compactifications of bounded $C_0$-semigroups}
\author{Josef Kreulich, Universit\"at Duisburg Essen}

\begin{abstract}
In this study, we refine the compactification presented by Witz \cite{Witz} for general semigroups to the case of bounded $C_0$-semigroups, herein involving adjoint theory for this class of operators. This approach considerably reduces the operator space in which the compactification is performed. Additionally, this approach leads to a decomposition of $X^{\sun}$ and to an extension of ergodic results to dual semigroups.
\end{abstract}

\maketitle
\section{Introduction}
In this study on the compactifications of bounded $C_0$-semigroups, we attempt to reduce the spaces in their construction. Rather than $L(X,X^{**})$, used by \cite{Witz}, this study shows that the compactification is part of a smaller space of operators, namely, 
$$
L_T(X,X^{\sun\sun}):=\bk{U\in L(X,X^{\sun\sun}): \  U^*(X^{\sun})\subset X^{\sun}, U^{\sun*}(X^{\sun\sun})\subset X^{\sun\sun}},
$$
where the spaces $X^{\sun}$ and $X^{\sun\sun}$ come with the underlying $C_0$-semigroup 
$\bk{T(t)}_{t\ge 0}.$ Furthermore, the compactification of \cite{Witz} leads to a connected compactification of the dual semigroup $\bk{T^{\sun}(t)}_{t\ge 0}.$ We show that  an algebra isomorphism maps the compactification of $\bk{T(t)}_{t\ge 0}$ on the one of $\bk{T^{\sun}(t)}_{t\ge 0}.$ 
Through the given approach, a decomposition of $X^{\sun}=X^{\sun}_a\oplus X^{\sun}_0$ is found, as well as of dual-space-valued uniformly continuous functions, 
as \cite{KnappDeco} did in the scalar-valued case using their algebra structure. In this scope, we apply methods similar to those used in the proofs of \cite{DeGli1} and \cite{DeGli2}. Furthermore, to obtain these results, we combine the abstract theory of right semitopological semigroups \cite{RuppertLNM}, and we compare the results of dual semigroups with those of \cite{HillePhillips} and \cite{neervenLNM}.

\insert\footins{\footnotesize The author wishes to thank Professor Ruess for his suggestions and advice.}

\section{The $\sun$-semigroup and the operator space $L_T(X,X^{\sun\sun})$ }
Throughout this study, $\scS:=\bk{T(t)}_{t\ge 0}$ denotes a bounded $C_0$-semigroup with the generator $A;$ we define
$$
X^{\odot}:=\bk{x^*\in X^*:\lim_{t\to 0}T^*(t)x^*=x^*} \mbox{ called X-sun }
$$
from \cite[Theorem 1.3.1]{neervenLNM}, and we find that $X^{\odot}$ is a closed, $w^*$-dense, and $T^*(t)$-invariant subspace. The $C_0-$semigroup on $X^{\sun}$ is denoted $\scS^{\sun}:=\bk{T^{\sun}(t)}_{t\in\rep}.$ Moreover, $X^{\odot}=\overline{D(A^*)}.$ For a given $x$, we define $O(x):=\bk{T(t)x}_{t\ge 0};$ if the element $x=x^{\sun}\in X^{\sun},$ we assume $O(x^{\sun})=\bk{T^{\sun}(t)x^{\sun}}_{t\ge 0}.$
Let $L(X,Y)$ denote the Banach space of bounded linear operators from $X$ to $Y.$ With this setting, we recall \cite[Definition 14.3.1.]{HillePhillips} for bounded operators $B\in L(X).$
\begin{defi}
\begin{enumerate}
\item Given a linear operator $B\in L(X),$ we denote $(B^*)_0$ as the restriction of $B^*$ to $X^{\sun},$ and we denote by $B^{\sun}$ the restriction of $B^*$ with domain $D(B^{\sun}):=\bk{x^*\in X^{\sun}:B^*x^*\in X^{\sun}}.$
\item For an operator $U\in L(X,X^{\sun\sun})$ with $U^*(X^{\sun})\subset X^{\sun}$, we define $U^{\sun*}:=(U^{\sun}_{|X^{ \sun}})^*.$
\end{enumerate}
\end{defi}

This leads in the $\odot$ context to the following set of operators:
\begin{equation}
L_T(X,X^{\odot\odot}):=\bk{U\in L(X,X^{\odot\odot}): \  U^*(X^{\odot})\subset X^{\odot}, U^{\sun*}(X^{\sun\sun})\subset X^{\sun\sun}}.
\end{equation}
The goal of this section is to show that the previously defined operator space is a Banach algebra, and the right and left translation has continuity properties in an operator topology close to the $w^*-$operator topology, which is defined while $X^{\sun\sun}$ is a subspace of a dual space.

Similar to \cite[pp.31-32]{neervenLNM}, let $i:X^{\sun}\to X^*$ be the inclusion, and let 
\begin{equation} \label{X-sun-star-embeds-to-X-star-star}
\Funk{r}{X^{**}}{X^{\sun*}}{x^{**}}{\bk{x^{\sun}\mapsto <x^{**},ix^{\sun}>}}
\end{equation}
be the restriction.
Further, let $j:X\to X^{**}$ be the natural embedding. Following the arguments in the second part of the proof of \cite[Theorem 2.4.2, pp.31-32]{neervenLNM}, we have the following:

\begin{pro} \label{X-embedded-to-X-sun-star} 
Let $X$ be a Banach space, and let $\bk{T(t)}_{t\ge 0}$ be a $C_0-$semigroup. Then,
\begin{enumerate}
\item $jX\subset X^{\sun*}$
\item $\overline{rjB_X}^{\sigma(X^{\sun*},X^{\sun})}=B_{X^{\sun*}}.$
\end{enumerate}
\end{pro}
\begin{proof}
For the first item, note that
\begin{eqnarray*}
<T^{\sun*}(t)jx,x^{\sun}>&=&<jx,T^{\sun}(t)x^{\sun}>=<x,T^{\sun}(t)x^{\sun}>\\
&=&<T(t)x,x^{\sun}>,
\end{eqnarray*}
which describes the embedding. 

For the second item, note that
$\overline{jB_X}^{\sigma(X^{**},X)}=B_{X^{**}},$ and $r(B_{X^{**}})=B_{\sun*}$ by a consequence of the Hahn-Banach theorem \cite[Thm. 11 ,p. 63]{DS}; hence, 
$\overline{rjB_X}^{\sigma(X^{\sun*},X^{\sun})}=B_{X^{\sun*}}.$
\end{proof}

Next, we show that with the multiplication defined below, $L_T(X,X^{\sun\sun})$ becomes a semigroup of operators.

For a given $U,V\in L_T(X,X^{\odot\odot})$, following \cite{Witz}, we define
$$
\Funk{U\circ V}{X}{X^{\sun*}}{x}{\bk{ x^{\odot }\mapsto <Vx,U^{\odot}x^{\odot}>}}.
$$

\begin{lem} \label{LT-semigroup-properties}
The given $C_0$-semigroup $\bk{T(t)}_{t\ge 0}$ itself is contained in $L_T(X,X^{\odot\odot}),$ and $T(t)T(s)=T(t)\circ T(s)$ for all $t,s\in\rep.$
Moreover, for a given $U,V\in L_T(X,X^{\odot\odot})$, we have $U\circ V\in L_T(X,X^{\odot\odot}).$
\end{lem}

\begin{proof}
By \cite[Theorem 1.3.1]{neervenLNM}, we have $T^*(t)X^{\sun}=T^{\sun}(t)X^{\sun}\subset X^{\sun}, $ and consequently, $T^{\sun *}(t)(X^{\sun\sun}\subset X^{\sun\sun}.$
Let $t,s\in\re,$ and $x\in X$; then, 
\begin{eqnarray*}
T(s)\circ T(t)x&=&\bk{x^{\sun}\mapsto <T(t)x,T^{\sun}(s)x^{\sun}>} \\
&=&\bk{x^{\sun}\mapsto <T(t+s)x,x^{\sun}> } \\
&=&T(t+s)x \mbox{ is viewed as a linear functional on } X^{\sun}.
\end{eqnarray*}

It remains to be proven that for given $U,V\in L_T(X,X^{\odot\odot})$, 
$U\circ V\in L_T(X,X^{\odot\odot}).$ 
First, it has to be verified that for all $x\in X,$ $(U\circ V)x \in X^{\sun*}.$ Note that for $x^{\sun}\in X^{\sun},$
\begin{eqnarray*}
\btr{<(U\circ V) x, x^{\sun}>}&=&\btr{<Vx,U^*x^{\sun}>}\le \nrm{Vx}\nrm{U^*}\nrm{x^{\sun}} \\
&=&\nrm{Vx}\nrm{U}\nrm{x^{\sun}},
\end{eqnarray*}
which verifies the first claim. Next, for $t >0, x\in X,$ $(U\circ V)x \in X^{\sun*}$, we prove the continuity in $0$ for the semigroup $\bk{T^{\sun\sun}(t)}_{t\ge 0}.$
\begin{eqnarray*}
T^{\sun\sun}(t)(U\circ V)x- (U\circ V)x &=& T^{\sun\sun}(t)\bk{x^{\sun}\mapsto <Vx,U^{\sun}x^{\sun}>}-(U\circ V)x  \\
&=& \bk{x^{\sun}\mapsto <Vx,U^{\sun}T^{\sun}(t)x^{\sun}>-<Vx,U^{\sun}x^{\sun}>} \\
&=&\bk{x^{\sun}\mapsto <U^{\sun *}Vx,T^{\sun}(t)x^{\sun}>-<U^{\sun *}Vx,x^{\sun}>} \\
&=&\bk{x^{\sun}\mapsto <T^{\sun\sun}(t)U^{\sun *}Vx,x^{\sun}>-<U^{\sun *}Vx,x^{\sun}>}. 
\end{eqnarray*}
Because $Vx\in X^{\sun\sun}$  and $U^{\sun *}(X^{\sun\sun})\subset X^{\sun\sun},$ 
we find that 
$$
\lim_{t\to 0}\sup_{\nrm{x^{\sun}}\le 1}\btr{<T^{\sun\sun}(t)U^{\sun *}Vx -U^{\sun *}Vx,x^{\sun}>}=0,
$$
 and we obtain $U\circ V \in L(X,X^{\sun\sun}).$ To prove $(U\circ V)^*_{|X^{\sun}}(X^{\sun})\subset X^{\sun}$, we compute 
\begin{eqnarray*}
T^{\sun}(t)(U\circ V)^*x^{\sun}&=&T^{\sun}(t)\bk{x\mapsto <x^{\sun},U\circ V x>} \\
&=&\bk{x\mapsto <x^{\sun},U\circ VT(t)x>} \\
&=&\bk{x\mapsto <T ^{\sun}(t)V^*U^*x^{\sun},x>}.
\end{eqnarray*}
Because $U,V\in L_T(X,X^{\sun\sun})$, we have $V^*U^*x^{\sun}\in X^{\sun},$ which proves $(U\circ V)^*_{|X^{\sun}}(X^{\sun})\subset X^{\sun}.$ 
Therefore, it remains to consider $(U\circ V)^{\sun *}(X^{\sun\sun}).$
Note that
\begin{eqnarray*}
<(U\circ V)x,x^{\sun}>&=& <Vx,U^{\sun}_{|X^{\sun}}x^{\sun}>\\
&=&<x,V^{\sun}_{|X^{\sun}}U^{\sun}_{|X^{\sun}}x^{\sun}>;
\end{eqnarray*}
applying Proposition \ref{X-embedded-to-X-sun-star}, we have
\begin{eqnarray*}
<(U\circ V)^{\sun *}x^{\sun\sun},x^{\sun}>&=& <x^{\sun\sun},(U\circ V)^{\sun}x^{\sun}> \\
&=&<x^{\sun\sun},V^{\sun}_{|X^{\sun}}U^{\sun}_{|X^{\sun}}x^{\sun}>.
\end{eqnarray*}
Consequently,
\begin{eqnarray*}
T^{\sun\sun}(t)(U\circ V)^{\sun *}x^{\sun\sun}&=&T^{\sun\sun}(t)\bk{x^{\sun}\mapsto<x^{\sun\sun},(U\circ V)^{\sun}x^{\sun}>} \\
&=&\bk{x^{\sun}\mapsto<x^{\sun\sun},(U\circ V)^{\sun}T^{\sun}(t)x^{\sun}>} \\
&=&\bk{x^{\sun}\mapsto<x^{\sun\sun},V^*_{|X^{\sun}}U^*_{|X^{\sun}}T^{\sun}(t)x^{\sun}>} \\
&=&\bk{x^{\sun}\mapsto<T^{\sun\sun}(t)(U^*_{|X^{\sun}})^* (V^*_{|X^{\sun}})^* x^{\sun\sun},x^{\sun}>} \\
&=&\bk{x^{\sun}\mapsto<T^{\sun\sun}(t)U^{\sun*}V^{\sun*}x^{\sun\sun},x^{\sun}>}. \\
\end{eqnarray*}
Now, the invariance assumptions $U^{\sun*}X^{\sun\sun}$ and $V^{\sun*}X^{\sun\sun}\subset X^{\sun\sun}$ serve for the proof.
\end{proof}
Next, we define some operator topologies.
\begin{defi} Let $X,Y$ be Banach spaces.
\begin{enumerate}
\item $w^*OT$ on $L(X,Y^*)$ is a net $\net{T}{\la}{\Lambda}\subset L(X,Y^*)$ that is convergent to $T\in L(X,Y^*)$ if
$$
\netlim{\la}{\Lambda}<T_{\la}x,y>=<Tx,y> \mbox{ pointwise on } x\in X, \ y \in Y.
$$ 
\item The topology $\kappa^{\sun\sun}$ on $L_T(X,X^{\sun\sun})$ is a net $\net{T}{\la}{\Lambda}\subset L_T(X,X^{\sun\sun})$ that is convergent to $T\in L_T(X,X^{\sun\sun})$ if
$$
\netlim{\la}{\Lambda}<T_{\la}x,x^{\sun}>=<Tx,x^{\sun}> \mbox{ pointwise on } x\in X, \ x^{*} \in X^{\sun}.
$$
\item The topology  $\kappa^{\sun}$ on $L(X^{\sun})$ is
a net $\net{T}{\la}{\Lambda}\subset L(X^{\sun})$ that is convergent to $T\in L(X^{\sun})$ if
$$
\netlim{\la}{\Lambda}<T_{\la}x^{\sun},x>=<Tx^{\sun},x> \mbox{ pointwise on } x\in X, \ x^{\sun} \in X^{\sun}.
$$
\end{enumerate}
\end{defi}
\begin{pro} \label{sun_lcs}
$\kappa^{\sun\sun}$ and $\kappa^{\sun}$ are Hausdorff and locally convex topologies on $L_T(X,X^{\sun\sun})$ and $L(X^{\sun})$, respectively.
\end{pro}
\begin{proof}
We start with $\kappa^{\sun\sun}.$ The convexity is straightforward. Hence, it remains to prove that if for $T\in L_T(X,X^{\sun\sun}),$ $<Tx,x^{\sun}>=0$ for all $(x,x^{\sun})\in X\times X^{\sun},$ then $T=0.$ Considering $Tx\in X^{\sun*}$ as a linear functional, we have that $Tx:X^{\sun}\to \ce$ is the null functional for all $x\in X.$ Hence, $Tx=0$ for all $x\in X$, which gives $T=0.$
To prove $\kappa^{\sun}$ is Hausdorff, let $T\in L(X^{\sun});$ then, $Tx^{\sun}$ can be viewed as an element of $X^*$ which vanishes on $X,$ and we have the same procedure as $\kappa^{\sun\sun}.$
\end{proof}
With the above definition, we have the following.
\begin{pro} \label{LT-semigroup-algebra}
\begin{enumerate}
\item $(L_T(X,X^{\odot\odot}),\circ)$ is a semigroup, and $(L_T(X,X^{\odot\odot}),+,\circ)$ is a Banach algebra \cite[Definition 10.1, pp. 227-228]{RudinFA} with respect to the canonical norm.
\item Let $V\in L_T(X,X^{\sun\sun})$ and $U\in L(X);$ then, 
$$
\Funk{R}{(L_T(X,X^{\sun\sun}),\kappa^{\sun\sun})}{(L_T(X,X^{\sun\sun}),\kappa^{\sun\sun})}{W}{V\circ W}
$$
and
$$
\Funk{L}{(L_T(X,X^{\sun\sun}),\kappa^{\sun\sun})}{(L_T(X,X^{\sun\sun}),\kappa^{\sun\sun})}{W}{W\circ U}
$$
are continuous.
\item If $U,V\in L_T(X,X^{\sun\sun})$ and $x\in X$ such that $Vx\in X$, then $(U\circ V)x=U(Vx).$
\end{enumerate}
\end{pro}
\begin{proof}
The first item is obvious by Lemma \ref{LT-semigroup-properties}. To prove the continuity claim, let $\net{W}{\gamma}{\Gamma}\subset L_T(X,X^{\sun\sun})$,
$\kappa^{\sun\sun}- \netlim{\gamma}{\Gamma}W_{\gamma}=W,$ $x\in X$ and $x^{\sun}\in X^{\sun};$  then,
\begin{eqnarray*}
<V\circ W_{\gamma}x,x^{\sun}>&=&<W_{\gamma}x,V^{\sun}x^{\sun}> 
\end{eqnarray*}
Because $V^{\sun}x \in X^{\sun}$, we obtain the continuity. For $L$, we have
\begin{eqnarray*}
<W_{\gamma}\circ U x, x^{\sun}>&=& <W_{\gamma}Ux,x^{\sun}>,
\end{eqnarray*}
and $Ux\in X$ serves for the proof. 
For the proof of the last item, 
let $U,V\in L_T(X,X^{\sun\sun})$ and $x\in X$ such that $Vx\in X$; then, for $x^{\sun}\in X^{\sun}$, we have
\begin{eqnarray*}
<(U\circ V)x,x^{\sun}>&=& <Vx,U^{*}_{|X^{*}}x^{\sun}> =<UVx,x^{\sun}>.
\end{eqnarray*}

\end{proof}

\section{Compactification}
The goal of this section is to follow the construction provided by \cite{Witz} and to show that the compactification stays in the smaller operator space $L_T(X,X^{\sun\sun}).$  Therefore, we use the original definition of $\circ.$ For a given $U,V\in L(X,X^{**})$, similar to \cite{Witz}, we define
$$
\Funk{U\circ V}{X}{X^{**}}{x}{\bk{ x^{*}\mapsto <Vx,U^{*}x^{*}>}}.
$$

If $j_{X^*}:X^*\to X^{***}$ denotes the natural embedding we have the mapping
\begin{equation}
\Funk{\eta}{L(X,X^{**})}{L(X^{*})}{U}{\bk{x^{*}\mapsto \eta(U)x^{*}:x\mapsto <x,U^*x^{*}>}=U^*j_{X^*}x^*},
\end{equation}
we find that
$$
\nrm{\eta(U)}=\sup_{x\in B_X}\sup_{x^{*}\in B_{X^{*}}}\btr{<Ux,x^{*}>}=\nrm{U}.
$$
Noting $V^*, U^*\in L(X^{***},X^*),$ for $U,V\in L(X,X^{**}),$ and natural embedding $j_{X^*}:X^*\to X^{***},$ we can define $V^*U^*$ as $V^*j_{X^*}U^*.$ Hence,
\begin{eqnarray} \label{algebra-structure}
\lefteqn{\eta(V)\eta(U)x^*=\eta(V)\bk{x\mapsto <x,U^*x^{*}>}} \\
&=&\eta(V)\bk{x\mapsto <x,j_{X^*}U^*x^{*}>} =\bk{x\mapsto <x,V^*j_{X^*}U^{*}x^{*}>} \nonumber \\
&=&\bk{x\mapsto<Vx,U^{*} x^{*}>} = \bk{x\mapsto <(U\circ V)x,x^{*}>}\nonumber \\
&=&\bk{x\mapsto <x,(U\circ V)^*j_{X^*}x^{*}>}=\eta(U\circ V)x^*, \nonumber 
\end{eqnarray}
which verifies an algebraic structure between the operator spaces and defined multiplications.
Endowing the operator spaces with the previously defined operator topologies, we find that
\begin{equation}
\Funk{\eta}{(L(X,X^{**}),w^*OT)}{(L(X^{*}),w^*OT)}{U}{\bk{x^{*}\mapsto \eta(U)x^{*}:x\mapsto <Ux,x^{*}>}}
\end{equation}
is an isomorphism. For $U\in L(X^*)$ and $U^*\in L(X^{**}),$ if $j:X\to X^{**}$ denotes the natural embedding, we claim 
$\eta(U^* j)=U.$
Because
$$ 
<x,Ux^*>=<jx,Ux^*>=<U^*(jx),x^*>=<x,\eta(U^*j)x^*>;
$$
hence, $\eta$ is surjective and together with (\ref{algebra-structure}) we obtain an algebra structure on $L(X,X^{**}).$

For the $w^*OT-w^*OT$ continuity of $\eta^{-1}$, let $\eta(U_{\al})\to \eta(U),$ and 
$$
\eta(U_{\al})x^*=\bk{x\mapsto <U_{\al}x,x^*>}\to\bk{x\mapsto <Ux,x^*>}.
$$ This is exactly the definition of the $w^*OT$ convergence in $L(X,X^{**}).$

Let $\bk{T(t)}_{t\ge0}=:\scS\subset L(X)\subset L(X,X^{**})$ be uniformly bounded by a constant $M$, and let $\scA=\eta(\scS).$ Because $\eta$ is an isometry, we have 
$$
\overline{A}^{\ w^*OT}\subset \Pi_{x\in X^{*}} \fk{M \nrm{x}B_{X^{*}},\sigma(X^{*},X)}.
$$
Hence, we obtain a compact $\scS_0, $
$$
\scS_0:=\eta^{-1}(\overline{A})\supset \scS.
$$
For the left and right multiplication in $(L(X,X^{**}),\circ))$, we have the following:
\begin{pro} If $U\in L(X)$ and $V\in L(X,X^{**})$, then
$$
\Funk{R}{(L(X,X^{**}),w^*OT)}{(L(X,X^{**}),w^*OT)}{W}{V\circ W}
$$
and
$$
\Funk{L}{(L(X,X^{**}),w^*OT)}{(L(X,X^{**}),w^*OT)}{W}{W\circ U}
$$
are continuous.
\end{pro}
\begin{proof}
To prove the continuity claim, let $\net{W}{\gamma}{\Gamma}\subset L(X,X^{**})$,
$w^*OT- \netlim{\gamma}{\Gamma}W_{\gamma}=W,$ $x\in X,$ and $x^*\in X^*$; then,
\begin{eqnarray*}
<V\circ W_{\gamma}x,x^{*}>&=&<W_{\gamma}x,V^{*}x^{*}>. 
\end{eqnarray*}
Because $V^{*}x^* \in X^{*}$, we obtain the continuity. For $L$, we have
\begin{eqnarray*}
<W_{\gamma}\circ U x, x^{*}>&=& <W_{\gamma}Ux,x^{*}>,
\end{eqnarray*}
and $Ux\in X$ serves for the proof. 
\end{proof}

Throughout this study, $\scS_0$ denotes the previously constructed compactification of $\scS.$  For this set, through the conclusions of \cite{Witz}, we have the following.

\begin{pro}[\cite{Witz}] \label{elemenatry-properties}
\begin{enumerate}
\item $\scS$ is $w^*OT$ dense in $\scS_0.$
\item $\scS_0$ is a semigroup, that is, for $U,V\in \scS_0$, we have $U\circ V \in  \scS_0.$
\item Let $U\in\scS_0$ and $t\ge 0;$ then, $T(t)\circ U=U\circ T(t).$
\end{enumerate}
\end{pro}

\begin{proof}
The compactness is a consequence of the construction. For denseness, note that 
$$
\funk{\eta}{\overline{\scA}}{\scS_0}
$$
is a homeomorphism. Therefore, let $T\in \scS_0;$ then, there exists a net $\net{S}{\la}{\Lambda}\subset \scA$ such that $\netlim{\la}{\Lambda}S_{\la}=\eta(T).$ Choose $T_{\la}=\eta^{-1}(S_{\la}).$
Next, we prove that $\scS_0$ is a semigroup.
Let $U,V\in \scS_0$ and $\net{W}{\gamma}{\Gamma},\net{U}{\la}{\Lambda}\subset \scS$ with 
$\netlim{\la}{\Lambda}U_{\la}=U$ and $\netlim{\gamma}{\Gamma}V_{\gamma}=V. $ Then,
\begin{eqnarray*}
<V\circ U x, x^{*}>&=& <Ux,V^{*}x^{*}> \\
&=&\netlim{\la}{\Lambda}<U_{\la}x,V^{*}x^{*}> \\
&=&\netlim{\la}{\Lambda}<VU_{\la}x,x^{*}> \\
&=&\netlim{\la}{\Lambda}\netlim{\gamma}{\Gamma}<V_{\gamma}U_{\la}x,x^{*}>.
\end{eqnarray*}
Hence,
$$
V\circ U= \netlim{\la}{\Lambda}\netlim{\gamma}{\Gamma} V_{\gamma}U_{\la},
$$
which proves $V\circ U\in\scS_0 .$  The fact that $T(t)$ commutes with $\scS_0$ is a consequence of the denseness and the continuity properties of $R_U,L_V.$
\end{proof}

We are ready to state our first main result that the compactification of a bounded $C_0-$semigroup embeds into $L_T(X,X^{\sun\sun}).$

\begin{theo} \label{LT-invariant-S0}

$\scS_0$ is $w^*OT$ compact in $L(X,X^{**}),$ and for $V\in\scS_0,$
\begin{enumerate} 
\item $V^*T^*(t)x^{\sun}=T^*(t)V^*x^{\sun}$, for all $x^{\sun}\in X^{\sun}$ 
\item $V^*(X^{\sun})\subset X^{\sun}$, and
\item $V^{\sun*}(X^{\sun\sun})\subset X^{\sun\sun};$ consequently, for all $x\in X$, we have $Vx\in X^{\sun\sun}.$
\end{enumerate}
Summarizing, we find that $\scS_0\subset L_T(X,X^{\sun\sun})$ and is $\kappa^{\sun\sun}$ compact.
\end{theo}
\begin{proof}

By Proposition \ref{elemenatry-properties} and $T(t)\circ V=V\circ T(t)$, we find for $x\in X,$ and $x^{\sun}\in X^{\sun},$ that $V^*x^{\sun}\in X^*$ and therefore,
\begin{eqnarray*}
<(T(t)\circ V)x,x^{\sun}>=<Vx,T^{\sun}(t)x^{\sun}>
 &=& <x, V^*j_{X^*}T^{\sun}(t)x^{\sun}>,
\end{eqnarray*}
and
\begin{eqnarray*}
<(V\circ T(t))x,x^{\sun}>=<T(t)x,V^*j_{X^*}x^{\sun}>
 &=& <x, T^*(t)V^*j_{X^*}x^{\sun}>.
\end{eqnarray*}
Thus,
\begin{eqnarray*}
\lim_{t\to 0}T^{*}(t)V^*j_{X^*}x^{\sun}&=&\lim_{t\to 0}V^*j_{X^*}T^{\sun}(t)x^{\sun}= V^*j_{X^*}x^{\sun},
\end{eqnarray*}
which proves that $V^*j_{X^*}(X^{\sun})\subset X^{\sun}.$ Consequently, $X^{\sun}\subset D(V^{\sun}j_{X^*}),$ the maximal domain, and we have  $V^{*}j_{X^*}\in L(X^{\sun})$ with $V^{*}j_{X^*}T^{\sun}(t)=T^{\sun}(t)V^{*}j_{X^*},$ for all $t\ge 0.$ Let $W:=V^{*}j_{X^*},$ The property (3), $W(X^{\sun\sun})\subset X^{\sun\sun}$ is a consequence of  \cite[Prop 2.4.1, Prop. 2.4.3.]{neervenLNM} applied to $\bk{T^{\sun}(t)}_{t\ge 0}, $ and $W.$

Because $\kappa^{\sun\sun}$ is the restriction of $w^*OT$ to $L_T(X,X^{\sun\sun})$, we obtain the $\kappa^{\sun\sun}$ compactness.
\end{proof}

\begin{remk} As the operators in the compactification of $\eta(\scS)$ are of the type $V^*j_{X^*},$ the operators in this compactification are not necessarily dual operators.
\end{remk}

Because the underlying space is an algebra, we can also consider $\scT:=co(\scS)$ and $\scU:=ac(\scS).$ The semigroup properties are straightforward; note that $((1-\la)+\la)((1-\mu)+\mu)=1$ for $\la,\mu\in [0,1].$ A similar computation proves it for the absolute convex hull, $ac(\scS).$
For $\scV\in \bk{\scT,\scU },$   we define the following.
\begin{defi} Let $\scT:=co(\scS)$ and $\scU:=ac(\scS).$  Then, 
\begin{eqnarray}
\scT_0&:=&\eta^{-1}(\overline{\eta(\scT)}^{w^*OT}), \\
\scU_0&:=&\eta^{-1}(\overline{\eta(\scU)}^{w^*OT}).
\end{eqnarray}
\end{defi}
As $\eta$ is an algebra isomorphism, we conclude, and $(M\nrm{x}B_{X^*})$ is absolutely convex and $w^*-$compact, we have
$$
\overline{A}^{\ w^*OT}\subset \overline{co A}^{\ w^*OT} \subset \overline{ac A}^{\ w^*OT} \subset \Pi_{x\in X^{*}} \fk{M \nrm{x}B_{X^{*}},\sigma(X^{*},X)}.
$$
are compact, and in consequence.

\begin{pro}
\begin{enumerate}
\item $\scS_0\subset\scT_0\subset\scU_0$.
\item $\scT$ is dense in $\scT_0,$ and $\scT_0$ is convex and compact.
\item $\scU$ is dense in $\scU_0,$ and $\scU_0$ is absolutely convex and compact.
\end{enumerate}
\end{pro}
As $\scT_0$ and $\scU_0$ are compactifications similar to the one constructed for $\scS,$ we have the following.

\begin{pro}[\cite{Witz}] \label{elementary-properties-T_0}
Let $\scV \in \bk{\scT,\scU};$ then,
\begin{enumerate}
\item $\scV_0$ is a semigroup; that is, for $U,V\in \scV_0$, we have $U\circ V \in  \scV_0.$
\item Let $U\in\scV_0$ and $t\ge 0;$ then, $T(t)\circ U=U\circ T(t).$
\end{enumerate}
\end{pro}
\begin{proof}
The proof is quite similar to that for $\scS_0$ in Proposition \ref{elemenatry-properties}.
\end{proof}
\begin{theo} \label{LT-invariant-T0} 
Let $\scV\in\bk{\scT,\scU}$ and $V\in\scV_0;$ then,
$\scV_0$ is $w^*OT$ compact in $L(X,X^{**});$ additionally,
\begin{enumerate} 
\item $V^*T^*(t)x^{\sun}=T^*(t)V^*x^{\sun}$, for all $x^{\sun}\in X^{\sun}$
\item $V^*(X^{\sun})\subset X^{\sun}$, and
\item $V^{\sun*}(X^{\sun\sun})\subset X^{\sun\sun}.$ Consequently, for all $x\in X$, we have $Vx\in X^{\sun\sun}.$
\end{enumerate}
Summarizing, we find that $\scV_0\subset L_T(X,X^{\sun\sun})$ and is $\kappa^{\sun\sun}$ compact.
\end{theo}
\begin{proof}
The proof is quite similar to that of Theorem \ref{LT-invariant-S0}.
\end{proof}

\section{The Influence of the Compactification on the Adjoint}
The boundedness of $\bk{T(t)}_{t\ge 0}$ implies the boundedness of $\bk{T^{\sun}(t)}_{t\ge 0};$ hence, we can repeat the compactification for the dual semigroup, but to keep the topologies connected, $\kappa^{\sun}$ is an adequate topology.

The purpose of this section is to obtain a connection between the compactifications of a semigroup and its $\sun-$semigroup. It is shown that we have a continuous algebra isomorphism in the defined operator topologies, which leads to an adequate compactification for the sun-dual-semigroup. Moreover, because the compactification of the dual semigroup is a set of endomorphisms, we derive a splitting of $X^{\sun}.$

In the following, let $\scS^{\sun}:=\bk{T^{\sun}(t)}_{t\in\rep},$ $\scT^{\sun}:=co(\scS^{\sun})$ and $\scU^{\sun}:=ac(\scS^{\sun}).$ By Proposition \ref{elemenatry-properties}, we learn that if 
$V\in L_T(X,X^{\sun\sun})$, then $V^*(X^{\sun})\subset X^{\sun},$  and $V^{\sun*}(X^{\sun\sun})\subset X^{\sun\sun}.$
Thus, we obtain by a mapping to the algebra 
$$
L_T(X^{\sun}):=\bk{T\in L(X^{\sun}):T^*(X^{\sun\sun})\subset X^{\sun\sun}},
$$
with
$$
\Funk{\eta^{\sun}}{L_T(X,X^{\sun\sun})}{L_T(X^{\sun})}{V}
{\bk{x^{\sun}\mapsto \eta^{\sun}V: x\mapsto <x,V^*_{|X^{\sun}}x^{\sun}>(=<Vx,x^{\sun}>)}}
$$
endowing $L_T(X^{\sun})$ with a $\kappa^{\sun}$ topology, which is a net $\net{V}{\gamma}{\Gamma}$ that converges to $V\in L_T(X^{\sun})$ if
$$
\netlim{\gamma}{\Gamma}<x,V_{\gamma}x^{\sun}>= <x,Vx^{\sun}> \mbox{ pointwise for } x\in X, x^{\sun}\in X^{\sun},
$$
$\eta^{\sun}$ is injective and continuous and $\eta(V\circ U)=U^{\sun}V^{\sun}.$ For $U\in L_T(X^{\sun}),$ $V:=U^*_{|X^{\sun\sun}}\in L(X^{\sun\sun}),$ and if $j:X\to X^{\sun\sun}$ denotes the natural embedding, we claim $\eta^{\sun}(\bk{x\mapsto V(jx)})=U.$
Because
$$ 
<x,Ux^{\sun}>=<jx,Ux^{\sun}>=<V(jx),x^{\sun}>=<x,\eta^{\sun}(Vj)x^{\sun}>,
$$
$\eta^{\sun}$ is surjective. Consequently, from the previous observations, we derive our second main result.
\begin{theo}
\begin{enumerate}
\item
$$
\Funk{\eta^{\sun}}{(L_T(X,X^{\sun\sun}),\kappa^{\sun\sun})}{(L_T(X^{\sun}),\kappa^{\sun})}{V}
{\bk{x^{\sun}\mapsto \eta^{\sun}V: x\mapsto <x,V^*_{|X^{\sun}}x^{\sun}>}}
$$
is continuous and an algebra isomorphism.
\item $\eta^{\sun}(T(t))=T^{\sun}(t).$ 
\item Defining
$$
\scS_0^{\sun}:=\eta^{\sun}(\scS_0), \ \scT_0^{\sun}:=\eta^{\sun}(\scT_0)
\mbox{ and  } \ 
\scU_0^{\sun}:=\eta^{\sun}(\scU_0),
$$
we have
\begin{enumerate}
\item $\scS^{\sun} \subset \scS_0^{\sun}\subset\scT_0^{\sun}\subset\scU_0^{\sun}\subset L(X^{\sun})$,
\item $\scS^{\sun}$ is dense in $\scS_0^{\sun},$ and $\scS_0^{\sun}$ is compact,
\item $\scT^{\sun}$ is dense in $\scT_0^{\sun},$ and $\scT_0^{\sun}$ is convex and compact, and
\item $\scU^{\sun}$ is dense in $\scU_0^{\sun},$ and $\scU_0^{\sun}$ is absolutely convex and compact,
\end{enumerate}
where $\kappa^{\sun}$ is the underlying topology.
\end{enumerate}
\end{theo}

\section{Ideal Theory}

The above construction opens up the possibility of applying the theory of compact right [left] topological semigroups \cite{RuppertLNM}. In this section, we show how this abstract theory applies to the $C_0-$semigroup.

Let $\scS,\scS_0$ be as in the previous section. A right [left] ideal of $\scS_0$ is a subset $I$ of $\scS_0$ 
such that $I\scS_0\subset I \ [\scS_0 I\subset I].$ The semigroup $\scS_0$ is a compact right  topological semigroup, i.e., 
$\scS_0$ is compact, and for a given $V\in\scS_0$, the translation
$$
\Funk{R}{(\scS_0,\kappa^{\sun\sun})}{(\scS_0,\kappa^{\sun\sun})}{W}{V\circ W}
$$
is continuous. 
The semigroup $\scS_0^{\sun}$ is a compact left topological semigroup, i.e., 
$\scS_0^{\sun}$ is compact, and for a given $V\in\scS_0$, the mapping
$$
\Funk{L}{(\scS_0^{\sun},\kappa^{\sun})}{(\scS_0^{\sun},\kappa^{\sun})}{W}{W V}
$$
is continuous. This gives the following for the considered compactifications:
\begin{lem} Let $\scV\in\bk{\scS,\scT,\scU}$; then,
\begin{enumerate}
\item $\scV_0$ is a compact right semitopological semigroup.
\item $\scV_0^{\sun}$ is a compact left semitopological semigroup.
\item $\eta^{\sun}$ is a semigroup isomorphism between $\scV_0$ and $\scV_0^{\sun}.$
\end{enumerate}
\end{lem}
\begin{theo}[\cite{Ellis}]
Every compact right [left] topological semigroup has an idempotent.
\end{theo} 
\begin{defi} [ {\cite[p. 12]{RuppertLNM}}]
The set of idempotents in a semigroup $S$ is denoted $E(S).$ We define relations $\le_L$ and $\le_R$ on E(S) by
\begin{eqnarray*}
e\le_L f &\mbox{ if } & ef=e, \\
e\le_R f &\mbox{ if } & fe=e.
\end{eqnarray*}
If $e$ and $f$ commute, then we omit the indices L and R.
\end{defi} 
\begin{defi}
Let $(A,\le)$ be a set with a transitive relation. 
Then, an element $a$ is called $\le-$maximal [$-$minimal] in $A$
 if for every $a^{\prime}\in A$, $a\le a^{\prime}$ implies $a^{\prime} \le a$ [$a^{\prime} \le a$ implies $a\le a^{\prime}$].
\end{defi}
Recalling \cite[p. 14]{RuppertLNM}, we have the following.
\begin{theo}
Every compact right topological semigroup contains $\le_L$-maximal and $\le_R-$minimal idempotents.
\end{theo}

\begin{theo}[{\cite[p. 21]{RuppertLNM}}] \label{right_topo_semigroups}
For an idempotent $e$ in a compact right topological semigroup $S$, the following statements are equivalent:
\begin{enumerate}
\item $e$ is $\le_R-$minimal in $E(S)$.
\item $e$ is $\le_L-$minimal in $E(S)$.
\item $eS$ is a minimal right ideal of $S$.
\item $eSe$ is a group, and $e$ is an identity in $eSe$.
\item $Se$ is a minimal left ideal of $S$.
\item $SeS$ is the minimal ideal of $S$.
\item S has a minimal ideal $M(S)$, and $e\in M(S).$
\end{enumerate}
\end{theo}

Next, we recall some definitions that entail certain compactness conditions on the orbit.
\begin{defi} Let $\jz\in \bk{\re,\rep,[a,\infty)}.$
\begin{enumerate}
\item A function $f\in C_b(\jz,X)$ is called Eberlein weakly almost periodic (E.-wap) if 
$$
O(f):=\bk{f_{\tau}:=\bk{\jz\ni t\mapsto f(t+\tau)}: \tau\in\jz}
$$
is weakly relatively compact in $C_b(\jz,X).$
Let 
\begin{eqnarray*}
W(\jz,X)&:=&\bk{f\in C_b(\jz,X): f \mbox{ is Eberlein weakly almost periodic}}, \\
W_0(\jz,X)&:=& \bk{f\in W(\jz,X): f_{t_n} \to 0 \mbox{ weakly for some } \seq{t}{n}\subset \jz }.
\end{eqnarray*}
\item A function $f\in C_b(\re,X)$ is called almost periodic if
$$
O(f):=\bk{f_{\tau}:=\bk{\re \ni t\mapsto f(t+\tau)}: \tau\in\re}
$$
is relatively compact in $C_b(\re,X).$
Let 
$$
AP(\re,X):=\bk{f\in C_b(\re,X): f \mbox{ is almost periodic}}.
$$
\end{enumerate}
\end{defi}

To provide a sufficient condition on $\bk{T(t)}_{t\in\rep}$ to identify the idempotent of Theorem \ref{right_topo_semigroups}, we recall some results on Eberlein weakly almost periodicity. 
\begin{theo}[\cite{RuessSemi},\cite{RuessAAP}]\label{RSdeco} 
Let $\bk{T(t)}_{t\in\rep}$ be a $C_0-$semigroup. 
Then, $\bk{t\mapsto T(t)x}$  is Eberlein weakly almost periodic iff for all $x\in X,$ $\bk{T(t)x:t\in\rep}$ is relatively weakly compact.
In the above case, we have $x=x_{ap}+x_0^E,$  $\bk{t\mapsto T(t)x_{ap}}$ is a restriction of an almost periodic function, and there exists a sequence $\seq{\om}{n}\subset \rep$ such that $\ilm{n}T(\cdot+\om_n)x_0^E= 0$ weakly in $BUC(\rep,X).$
\end{theo}

Next, we apply the compactification and results from the adjoint semigroup.

\begin{theo} \label{apply-to-sun}
Let $\bk{T(t)}_{t\in\rep}$ be a bounded semigroup, and let $P^{\sun}$ denote a minimal idempotent in $\scS_0^{\sun}$ given by Theorem \ref{right_topo_semigroups}; then, $X^{\sun}$ decomposes into a direct sum of two closed and translation-invariant subspaces  $R(P^{\sun})=:X^{\sun}_a$ and $N(P^{\sun})=:X^{\sun}_0.$ Moreover, we have
\begin{enumerate}
\item $  P^{\sun}\scS_0^{\sun}$ is a group on $X_a^{\sun}$.
\item $x^{\sun}\in X^{\sun}_a$ iff for every $V\in \scS_0^{\sun}$, there exists an $U\in \scS_0^{\sun}$ with $P^{\sun}UP^{\sun}Vx^{\sun}=x^{\sun}$.
\item If $x^{\sun}\in X^{\sun}_0,$ then there exists a net $\net{t}{\gamma}{\Gamma}$ such that $\sigma(X^{\sun},X)-\netlim{\gamma}{\Gamma}T^{\sun}(t_{\gamma})x^{\sun}=0.$
\item Let $ x^{\sun}\in X^{\sun}$ such that $\bk{t\mapsto T^{\sun}(t)x^{\sun}}$ is almost periodic; then, $x^{\sun}\in X_a^{\sun}.$
\item Let $ x^{\sun}\in X^{\sun}$ and, for a net $\net{t}{\al}{A},$ $\sigma(X^{\sun},X^{\sun*})-\netlim{\al}{A}T^{\sun}(t_{\al})x^{\sun} =0;$ then, $x^{\sun}\in X_0.$
\item \label{X_a-orthogonal} 
Let $x\in (X^{\sun}_a)_{\perp}:=\bk{x\in X: <x,x^{\sun}>=0 \ \forall x^{\sun}\in X^{\sun}_a};$ then, there is a net $\net{t}{\al}{A}$ such that
$$\sigma(X,X^{\sun})-\netlim{\al}{A} T(t_{\al})x=0.$$
\item \label{X_0-orthogonal}
Let $x \in (X^{\sun}_0)_{\perp}:=\bk{x\in X: <x,x^{\sun}>=0 \ \forall x^{\sun} \in X^{\sun}_0};$ then, there is a net $\net{t}{\al}{A}$ such that
$$\sigma(X,X^{\sun})-\netlim{\al}{A} T(t_{\al})x=x.$$
\item \label{sun-subspace-split} Let $Y\subset X^{\sun}$ be a closed subspace and $x^{\sun}\in Y.$ If $\overline{O(x^{\sun})}^{\sigma(X^{\sun},X)}\subset Y,$ then $x^{\sun}_a,x^{\sun}_0\in Y.$ Consequently, $Y=Y_a\oplus Y_0,$ with $Y_a:=Y \cap R(P^{\sun})$ and $Y_0:=Y\cap N(P^{\sun}).$
\end{enumerate}
\end{theo}

\begin{proof}
By Theorem \ref{right_topo_semigroups}, we find a minimal idempotent $e=:P^{\sun}\in \scS_0^{\sun}$, as $(P^{\sun})^2=P^{\sun}$ and is bounded, and it is a continuous projection, which serves for the decomposition. The translation invariance entails $T^{\sun}(t)P^{\sun}=P^{\sun}T^{\sun}(t).$
The first claim is a direct consequence of Theorem \ref{right_topo_semigroups} (4) and $P^{\sun}X_a^{\sun}=X_a^{\sun}.$

To prove (2), note that $P^{\sun}\scS_0^{\sun}P^{\sun}$ is a group; hence, for a given $V\in\scS_0^{\sun}, $ we find an operator $W\in\scS_0^{\sun}$ such that 
$ (P^{\sun}WP^{\sun})P^{\sun}VP^{\sun}=P^{\sun}.$ For $x^{\sun}\in X_a^{\sun}$, this leads to
$(P^{\sun}WP^{\sun})P^{\sun}VP^{\sun}x^{\sun}=P^{\sun}x.$ The choice of $x^{\sun}$ leads to $P^{\sun}x^{\sun}=x^{\sun},$ and $(P^{\sun})^2=P^{\sun}$ leads to 
$U:=P^{\sun}WP^{\sun}.$ 
For the other direction, note that $x=P^{\sun}UP^{\sun}Vx\in X_a^{\sun}.$

Let $x^{\sun}\in N(P^{\sun});$ then, $0=Px^{\sun}=\netlim{\al}{A}T(t_{\al})x^{\sun}$ for an appropriate net $\net{t}{\al}{A}\subset \re.$

Let $x^{\sun}\in X^{\sun}$ such that $\bk{t\mapsto T^{\sun}(t)x^{\sun}}$ is almost periodic. Because of the decomposition, we find $x^{\sun}=x^{\sun}_a+x^{\sun}_0$, with $x^{\sun}_a\in X^{\sun}_a$ and $x_0\in X^{\sun}_0.$ By the almost periodicity, we have that $\bk{t\mapsto P^{\sun}T^{\sun}(t)x}$ is almost periodic. Because the semigroup commutes with $P^{\sun},$ we have 
\begin{equation}
\bk{t\mapsto P^{\sun}T^{\sun}(t)x^{\sun}}=\bk{t \mapsto T^{\sun}(t)P^{\sun}x^{\sun}}= \bk{t\mapsto T^{\sun}(t)x^{\sun}_a}.
\end{equation}
Therefore, $\bk{t\mapsto T^{\sun}(t)x^{\sun}_a}$ is almost periodic, which, according to \cite{RuessAAP}, leads to a decomposition of 
$x^{\sun}_a=x^{\sun}_{ap}+x_0^{1},$ with $\bk{t\mapsto T(t)x^{\sun}_{ap}}$ being almost periodic 
and $\bk{t\mapsto T(t)x_0^1}$ in $C_0(\rep,X).$ 
By the almost periodicity of $\bk{t\mapsto T^{\sun}(t)x^{\sun}_a},$ $x_0^{1}=0.$ In summary, $\bk{t\mapsto T^{\sun}(t)x^{\sun}_0}$ is almost periodic, and the orbit becomes norm compact; hence, $\nrm{T(t+t_{\gamma})x_0^{\sun}}\le C\nrm{T(t_{\gamma})x_0^{\sun}},$ and $0$ is a cluster point of $O^+(\bk{t\mapsto T^{\sun}(t)x^{\sun}_0})$, which yields $x^{\sun}_0=0.$

Let $x^{\sun}\in X^{\sun},$ $x^{\sun}=x^{\sun}_a+x^{\sun}_0$ and, for a net $\net{t}{\al}{A},$
$T^{\sun}(t_{\al})x^{\sun}\to 0$ weakly; then,
$$
T^{\sun}(t_{\al})x_a^{\sun}=T^{\sun}(t_{\al})P^{\sun}x^{\sun}=P^{\sun}T^{\sun}(t_{\al})x^{\sun}\to 0 \mbox{ weakly }.
$$
Let $Q=\kappa^{\sun}-\netlim{\al}{A}T(t_{\al})$; then, $Qx_a=0$. Using $P^{\sun}\scS_0^{\sun}$ as a group on $X_a$, we find $x_a=0;$ hence, $x^{\sun}=x_0.$

Let $\net{t}{\al}{A}$ be a net such that $P^{\sun}=\netlim{\al}{A}T(t_{\al}).$ Then, for $x\in (X^{\sun}_0)_{\perp},$ we have
\begin{eqnarray*}
0&=&<x,P^{\sun}x^{\sun}>=\netlim{\al}{A}<x,T^{\sun}(t_{\al})x^{\sun}> \\
&=& \netlim{\al}{A}<T(t_{\al})x,x^{\sun}>, \mbox{ for all } x^{\sun}\in X^{\sun},
\end{eqnarray*}
which proves the $\sigma(X,X^{\sun})$ convergence.

Let $\net{t}{\al}{A}$ be a net such that $P^{\sun}=\netlim{\al}{A}T^{\sun}(t_{\al}).$ Then, for $x\in (X^{\sun}_a)_{\perp},$ we have
\begin{eqnarray*}
0&=&<x,(I-P^{\sun})x^{\sun}>=\netlim{\al}{A}<x,(I-T^{\sun}(t_{\al}))x^{\sun}> \\
&=& \netlim{\al}{A}<(I-T(t_{\al}))x,x^{\sun}>, \mbox{ for all } x^{\sun}\in X^{\sun},
\end{eqnarray*}
which proves the $\sigma(X,X^{\sun})$ convergence.

Because $x^{\sun}_a=\sigma(X^{\sun},X)-\netlim{\al}{A}T(t_{\al})x^{\sun}\in \overline{O(x^{\sun})}^{\sigma(X^{\sun},X)}\subset Y,$ for some appropriate net, the proof is completed.
\end{proof}

\begin{remk} To obtain weak convergence in (\ref{X_a-orthogonal}) and (\ref{X_0-orthogonal}), the orbit $\bk{T(t)x}_{t\ge0}$ has to be weakly equicontinuous; compare with \cite[Cor. 2.2.4, p. 26]{neervenLNM}.
\end{remk}

\begin{remk} \label{sub_semigroups}
If $Y\subset X^{\sun}$ is a closed subspace such that $\overline{O(x^{\sun})}^{\sigma(X^{\sun},X)}\subset Y$ for all $x^{\sun}\in Y,$ then ${\scS^{\sun}_0}_{|Y}$ is a compact left topological semigroup. Clearly, if $\scS^{\sun}_0$ is commutative, then ${\scS^{\sun}_0}_{|Y}$ is commutative.
\end{remk}

\begin{cor}
If $x^{\sun}\in X^{\sun}$  such that $\bk{t\mapsto T^{\sun}(t)x^{\sun}}$ is Eberlein weakly almost periodic, then the E.-wap splitting and the one from Theorem \ref{apply-to-sun} coincide.
\end{cor}
By the previous lemma, we conclude the following.
 
\begin{cor}
If $E=\bk{P^{\sun}\in E(\scS^{\sun}_0): \ \le_L\mbox{-minimal}}$,
\begin{eqnarray*}
X_{ap}^{\sun}&:=&\bk{x^{\sun}\in X^{\sun}: \bk{t\mapsto T^{\sun}(t)x^{\sun}} \mbox{ is almost periodic }}\subset \bigcap_{P^{\sun}\in E }R(P^{\sun}), \\
&&\bk{x^{\sun}\in X^{\sun}: 0\in \overline{\bk{T^{\sun}(t)x:t\ge 0}}^{\sigma(X^{\sun},X^{\sun*})}} \subset \bigcap_{P^{\sun}\in E }N(P^{\sun}).
\end{eqnarray*}
\end{cor}

\section{Topological and Algebraic Properties on Given Vectors}
In this section, we show how even slight assumptions of weak compactness on the orbit lead to the algebraic structure of the compactification. The main results provide a necessary and sufficient condition on the $C_0-$semigroup for two operators to commute on given vectors $x$ and $x^{\sun}$, separately.  

The previous work serves as a necessary refinement of the discussion of the following sets.  
\begin{eqnarray*}
\Xsr &:=&\bk{x^{\sun}\in X^{\sun}: \mbox{ For all } V\in\scS_0^{\sun} \mbox{ there exists a } U\in \scS_0^{\sun} \mbox{ such that } UVx^{\sun}=x^{\sun} }, \mbox{ and } \\
X_{fl}^{\sun}&:=& \bk{x^{\sun}\in X^{\sun}: \mbox{ there exists a net } \net{s}{\al}{A} \mbox{ such that } w^*-\netlim{\al}{A}T^{\sun}(t_{\al})x^{\sun}=0 \ } \\
&=&\bk{x^{\sun}\in X^{\sun}: \mbox{ there exists a } V\in\scS_0^{\sun} \mbox{ such that } Vx^{\sun}=0 \ }.
\end{eqnarray*}
These sets are discussed in several contexts; we refer to \cite{Krengel}, and especially regarding the theory of topological dynamics, we refer to \cite{Shen}. 
In general, $X_{rev}$ need not be a vector space \cite[p. 7 Exa. 2.8]{BasitGuenzler_recurr}, but we want to provide an answer when $\Xsr=X_a^{\sun}$ and $X_{fl}^{\sun}=X_0^{\sun}$.
\begin{defi}
A vector $x^{\sun}\in \Xsr$ is called $\kappa^{\sun}$ reversible, and a vector $x^{\sun}\in X_{fl}^{\sun}$  is a $\kappa^{\sun}-$flight vector.
\end{defi}

\begin{pro} \label{recurr-results}
\begin{enumerate}
\item $\Xsr\cap X_{fl}^{\sun}=\{0\}.$
\item $\Xsr$ and $X_{fl}^{\sun}$ are norm closed.
\item If $\scS^{\sun}_0x^{\sun}\subset X_a^{\sun}$ for all $x^{\sun}\in X_a^{\sun},$ then $X_a^{\sun}=\Xsr.$
\end{enumerate}
\end{pro}
\begin{proof}
Clearly, $\Xsr\cap X_{fl}^{\sun} = \bk{0}.$ 
Furthermore, $\Xsr$ is closed. Let $\seq{y}{k}\subset \Xsr$ and $\ilm{k}\nrm{y_k-y}=0$ for some $y\in X^{\sun}$. Then, for $V\in \scS_0^{\sun}$, there exists $U_n\subset \scS_0^{\sun}$ such that
$$
U_nVy_n=y_n \mbox{ for all } n\in\za.
$$
Let $\net{t}{\beta}{B}\subset \za $ be a subnet such that
$\kappa^{\sun}-\netlim{\beta}{B}U_{\beta}=U.$ Then, we have the following:

\begin{eqnarray*}
\btr{<UVy-y,x>}&=&\btr{<UVy-U_{\beta}Vy,x>+ <U_{\beta}Vy-U_{\beta}Vy_{\beta},x>}\\
&&+<U_{\beta}Vy_{\beta}-y_{\beta},x>+<y_{\beta}-y,x> \\
&\le&\btr{<UVy-U_{\beta}Vy,x>}+2 C \nrm{y-y_{\beta}},
\end{eqnarray*}
which proves the claim. The proof for $X_0^{\sun}$ is quite analogous.

We have $X_a^{\sun}\subset \Xsr$  because for given 
$L\in\scS^{\sun}_0,$ we have $P^{\sun}L\in \scS^{\sun}_0$, and we find 
$P^{\sun}S$, with $P^{\sun}SP^{\sun}L=P^{\sun}$ on $X_a^{\sun}.$
Therefore, it remains to show that $\Xsr\subset X_a^{\sun}.$ Let $x\in \Xsr;$ then, $x=x_a\oplus x_0,$  whereby the splitting is given with some minimal idempotent $P^{\sun}.$ As $x\in\Xsr$, we find some $U\in\scS_0^{\sun}$ such that $x=UP^{\sun}x=UP^{\sun}x_a.$ By the assumption $\scS^{\sun}_0y^{\sun}\subset X_a^{\sun}$ for all $y^{\sun}\in X_a^{\sun},$ we find $x\in X_a^{\sun}.$
\end{proof}

\begin{lem}
\begin{enumerate}
\item If $P^{\sun*}(X)\subset X,$ then $\Xsr=X_a^{\sun}$, $X_{fl}^{\sun}=X_0^{\sun},$  and $P^{\sun}V=VP^{\sun}$ for all $V\in\scS_0^{\sun},$ i.e., $X_a^{\sun}$, $X_0^{\sun}$, are $\scS_0^{\sun}$ invariant. The minimal idempotent in the semigroup $\scS_0^{\sun}$ is unique.
\item If $P^{\sun*}(X)\subset X,$ then $P:=P^{\sun*}_{|X}=\sigma(X,X^{\sun})-OT-\netlim{\al}{A}T(t_{\al})$ for an appropriate net $\net{t}{\al}{A}.$
\end{enumerate}
\end{lem}

\begin{proof}
First, we verify that $VP^{\sun}=P^{\sun}V$ for all $V\in \scS_0^{\sun}.$
If $P(X)\subset X$ and $V\in\scS_0^{\sun},$ then we have for a net $\net{t}{\al}{A}$,
\begin{eqnarray*}
<x,VP^{\sun}x^{\sun}>&=&\netlim{\al}{A}<x,T(t_{\al})P^{\sun}x^{\sun}>\\
&=&\netlim{\al}{A}<x,P^{\sun}T(t_{\al})x^{\sun}>=<Px,Vx^{\sun}>=<x,P^{\sun}Vx^{\sun}>.
\end{eqnarray*}
Let $x\in \Xsr;$ then, for $P^{\sun}$, there exists a $V$ such that $x=VP^{\sun}x=P^{\sun}Vx\in X_a.$ Let $x\in X_{fl}^{\sun}\subset X^{\sun}=X_a^{\sun}\oplus X_0^{\sun},$ i.e., $x=x_a+x_0.$ Let $\net{t}{\al}{A}$ be the associated net to the flight vector $x$. Without loss of generality, $\netlim{\al}{A}T(t_{\al})x=U.$
Then,
$Ux_a=-Ux_0,$ and we find $V\in \scS_0^{\sun}$ such that $x_a=PVUx_a=-PVUx_0=-VUPx_0=0.$ Hence, $x=x_0.$ 
\end{proof}

\begin{cor} If $\scS_0^{\sun}$ is Abelian, then $X^{\sun}_a=\Xsr$ and $X^{\sun}_0=X^{\sun}_{fl}.$ Moreover, $\scS_0^{\sun}$ is a group on $X^{\sun}_a.$
\end{cor}

Next, we will provide a necessary and sufficient condition for $\scS_0^{\sun}$ to be Abelian. Therefore, we provide the next proposition.

\begin{pro} \label{interchanged1}
Let $f:\rep\to\ce$ E.-wap and $\net{t}{\la}{\Lambda},
\net{s}{\gamma}{\Gamma}\subset \rep.$ 
Then, we may pass to subnets $\bk{s_{\gamma_{\alpha}}}_{\alpha\in A}$ and $\bk{t_{\la_{\beta}}}_{\beta\in b}$ such that the iterated limits 
\begin{eqnarray*}
\nu&=& \netlim{\alpha}{A}\netlim{\beta}{B} f(t_{\la_{\beta}}+s_{\gamma_{\al}}) \mbox{ and} \\
\mu&=& \netlim{\beta}{B}\netlim{\al}{A}f (t_{\la_{\beta}}+s_{\gamma_{\al}})
\end{eqnarray*}
exist, and we have $\nu=\mu.$
\end{pro}
\begin{proof}
Because $f$ is Eberlein weakly almost periodic, $\bk{f_{t_{\la}}}_{\la\in\Lambda}$ is relatively weakly compact and 
$\bk{\delta_{s_{\gamma}}}_{\gamma\in\Gamma}$ is relatively $w^*$ compact, we may pass to convergent subnets. Using $f(t_{\la}+s_{\gamma})=\delta_{s_{\gamma}}f_{t_{\la}}$, we find that the iterated limits exist and that they are equal.
\end{proof}

The next theorem shows how Eberlein weak almost periodicity can be used to verify an Abelian structure for the compactification.
\begin{theo} \label{pointwise_Abelian}
Let $\scV\in\bk{\scS,\scT,\scU};$ then
\begin{enumerate}
\item Let $x\in X $ and $\bk{T(t)}_{t\in\rep}$ be a bounded $C_0-$semigroup. Then,
$$
\bk{ t\mapsto <T(t)x,x^{\sun}>}\in W(\rep) \mbox{ for all } x^{\sun}\in X^{\sun}
$$
if and only if $(U\circ V)x=(V\circ U)x$ for all $U,V\in \scV_0.$
\item Let $x^{\sun}\in X^{\sun} $ and $\bk{T(t)}_{t\in\rep}$ be a bounded $C_0-$semigroup. Then,
$$
\bk{ t\mapsto <x,T^{\sun}(t)x^{\sun}>}\in W(\rep) \mbox{ for all } x\in X
$$
if and only if $(UVx^{\sun}=(V U)x^{\sun}$ for all $U,V\in \scV_0^{\sun}.$
\end{enumerate}
\end{theo}

\begin{proof} For simplicity, we start with $\scS_0.$ Let
$$\bk{ t\mapsto <T(t)x,x^{\sun}>}\in W(\rep) \mbox{ for all } x^{\sun}\in X^{\sun}, 
$$
and given $U,V\in \scS_0$, let 
$\net{t}{\la}{\Lambda},\net{s}{\gamma}{\Gamma}\subset \rep$ be the corresponding nets such that 
$V=\netlim{\la}{\Lambda}T(t_{\la}) $ and 
$U=\netlim{\gamma}{\Gamma}T(s_{\gamma}).$ 
Then, for $x^{\sun}\in X^{\sun},$
\begin{eqnarray*}
<(U\circ V)x,x^{\sun}>&=&\netlim{\la}{\Lambda}<T(t_{\la})x,U^{\sun}x^{\sun}> \\
&=&\netlim{\la}{\Lambda}\netlim{\gamma}{\Gamma}<T(t_{\la}+s_{\gamma})x,x^{\sun}> \\
&& \mbox{applying Prop. \ref{interchanged1}} \\
&=&\netlim{\gamma}{\Gamma}\netlim{\la}{\Lambda}<T(t_{\la}+s_{\gamma})x,x^{\sun}>\\
&=&\netlim{\gamma}{\Gamma}<T(s_{\gamma})x,V^{\sun}x^{\sun}> \\
&=& <(V\circ U)x,x^{\sun}>.
\end{eqnarray*}
To obtain the backward implication, apply, for the given $\seq{t}{n},\seq{s}{m}\subset \rep$, the compactness $\scS_0$. Hence, we find subnets such that
$V=\netlim{\la}{\Lambda}T(t_{n_{\la}})$ and 
$U=\netlim{\gamma}{\Gamma}T(s_{m_{\gamma}}).$ 
The assumption $U\circ V=V\circ U$ verifies the criterion of \cite{Groth}. 

Now, let $U,V\in\scU_0.$ Then, we find nets 
$\bk{t_i^{\gamma}}_{i\in\za,\gamma\in\Gamma},\bk{s_i^{\lambda}}_{i\in\za,\lambda\in\Lambda}\subset\rep$ 
and $\bk{\alpha_i^{\gamma}}_{i\in\za,\gamma\in\Gamma},\bk{\beta_i^{\lambda}}_{i\in\za,\lambda\in\Lambda}\subset \re$, 
with  
$\sum_{i=1}^{n_{\gamma}} \btr{\alpha_i^{\gamma}}\le1$
and 
$\sum_{i=1}^{m_{\lambda}}\btr{\beta_i^{\lambda}}\le 1,$ such that

$$
U_{\gamma}:=\sum_{i=1}^{n_{\gamma}}\alpha_i^{\gamma}T(t_i^{\gamma})\mbox{ with } \netlim{\gamma}{\Gamma}U_{\gamma}=U \mbox{ and }
V_{\lambda}=\sum_{i=1}^{m_{\lambda}}\beta_i^{\lambda}T(s_i^{\lambda}), \mbox{ with } \netlim{\lambda}{\Lambda}V_{\lambda}=V.
$$
Now, define $g(t):=<T(t)x,x^{\sun}>,$ which is assumed to be Eberlein weakly almost periodic, and the bounded linear functional $\delta_t(g):=g(t); $ then, the duality reads as
\begin{eqnarray*}
<U\circ Vx,x^{\sun}>&=&\netlim{\lambda}{\Lambda}<V_{\lambda}x,U^{\sun}x^{\sun}>\\
&=&\netlim{\lambda}{\Lambda}\netlim{\gamma}{\Gamma}
\sum_{i=1}^{n_{\gamma}}\sum_{j=1}^{m_{\lambda}}\alpha_i^{\gamma}\beta_j^{\lambda}<T(t_i^{\gamma}+s_{j}^{\lambda})x,x^{\sun}>,\\
&=&\netlim{\lambda}{\Lambda}\netlim{\gamma}{\Gamma}
\sum_{i=1}^{n_{\gamma}}\sum_{j=1}^{m_{\lambda}}\alpha_i^{\gamma}\beta_j^{\lambda}g(t_i^{\gamma}+s_{j}^{\lambda}), \\
&=&\netlim{\lambda}{\Lambda}\netlim{\gamma}{\Gamma}
<\sum_{i=1}^{n_{\gamma}}\alpha_i^{\gamma}g(\cdot+t_i^{\gamma}),\sum_{j=1}^{m_{\lambda}}\beta_j^{\lambda}\delta_{s_j^{\lambda}}>_{(BUC(\rep),BUC(\rep)^*)}\\
&=&\netlim{\gamma}{\Gamma}\netlim{\lambda}{\Lambda}
<\sum_{i=1}^{n_{\gamma}}\alpha_i^{\gamma}g(\cdot+t_i^{\gamma}),\sum_{j=1}^{m_{\lambda}}\beta_j^{\lambda}\delta_{s_j^{\lambda}}>_{(BUC(\rep),BUC(\rep)^*)}.
\end{eqnarray*}
As $O(g)$ is weakly relatively compact in $BUC(\rep),$ 
its closed absolutely convex hull is weakly compact. Further, because $\nrm{\delta_t}\le 1,$ the absolute convex combination is bounded. Hence, we have separated the limits and obtain that the interchanged limits coincide. Doing the backward computation, we obtain the claim. To prove it for the dual semigroup, apply $\eta^{\sun}. $
\end{proof}

\begin{cor} \label{all_abelian}
If $\scS_0$ is Abelian, then so are $\scT_0$ and $\scU_0$.
\end{cor}

The previous strong result leads by an application of \cite{Groth} to the following theorem. It serves in obtaining an ergodic result for the dual semigroup from the original semigroup on $X$ and vice versa.

\begin{theo} \label{S0-Abelian}
Let $\bk{T(t)}_{t\in\rep}$ be a bounded $C_0-$semigroup and $\scV\in \bk{\scS,\scT,\scU};$ then, the following are equivalent:
\begin{enumerate}
\item $\bk{t\mapsto <T(t)x,x^{\sun}>}\in W(\rep)$ for all $x\in X, x^{\sun}\in X^{\sun}.$
\item $\scV_0$ is Abelian.
\item $\bk{t\mapsto <x,T^{\sun}(t)x^{\sun}>}\in W(\rep)$ for all $x\in X, x^{\sun}\in X^{\sun}.$
\item $\scV^{\sun}_0$ is Abelian.
\end{enumerate}
\end{theo}
\begin{remk}
The above result that $P^{\sun}$ commutes with every operator will not necessarily lead to Eberlein weak almost periodicity, as shown in Example \ref{weak-null-but-not-Eberlein} and Example \ref{Abelian_but_not_Eberlein}. 
\end{remk}

Now, a few results from Jacobs-Deleeuw-Glicksberg are obtained by the above.
\begin{theo}[{\cite[pp. 103-106]{Krengel}}] \label{JDG} If $\bk{T(t)}_{t\in\rep}$ is Eberlein weakly almost periodic, then
$\scS_0$ is an Abelian semigroup on $X$ and an Abelian group on $X_{ap}.$ Consequently, we have in the underlying case that $\scS_0^{\sun}$ is Abelian.
\end{theo}

Next, we show how the semigroup and its sun-dual are connected if the semigroup is Eberlein weakly almost periodic. 

\begin{theo} 
If $\bk{T(t)}_{t\in\rep}$ is Eberlein weakly almost periodic, then 
$X=X_{ap}\oplus X_0$, with a projection $\funk{V}{X}{X}$ satisfying $V(X)=X_{ap}.$ For the dual semigroup, we have  
$X_a^{\sun}=\Xsr, \ X_0^{\sun}=X_{fl}^{\sun},$ with $X^{\sun}=\Xsr\oplus X_{fl}^{\sun},$ with a projection $\funk{P^{\sun}}{X^{\sun}}{X^{\sun}}$ satisfying $P^{\sun}(X^{\sun})=X^{\sun}_a.$ In this setting, we have $P^{\sun}=\eta^{\sun}(V),$ and the minimal idempotent is unique.
\end{theo}

\begin{proof} It suffices to verify that $P^{\sun*}(X)\subset X. $ 
By Theorem \ref{RSdeco}, we find that $X=X_{ap}\oplus X_0.$ Let $V$ be the corresponding projection and $V^{\sun}:=\eta(V).$ Furthermore, let $X^{\sun}=X^{\sun}_a\oplus X^{\sun}_0,$ and let $P^{\sun}$ be the corresponding minimal idempotent. We define $P:=\eta^{-1}(P^{\sun}).$ Then,
\begin{eqnarray*}
<x,V^{\sun}V^{\sun}x^{\sun}>&=&<Vx,V^{\sun}x^{\sun}>=<V\circ V x,x^{\sun}> \\
&=&<Vx,x^{\sun}>=<x,V^{\sun}x^{\sun}>,
\end{eqnarray*}
and for $P$, we have
\begin{eqnarray*}
<(P\circ P)x,x^{\sun}>&=& <Px,P^{\sun}x^{\sun}>=<x,P^{\sun}P^{\sun}x^{\sun}>\\
&=&<x,P^{\sun}x^{\sun}>=<Px,x^{\sun}>.
\end{eqnarray*}
Hence, we have that $P$ and $V^{\sun}$ are idempotents in $\scS_0$ and $\scS_0^{\sun}$.

By Theorem \ref{right_topo_semigroups}, we have that $V$ is minimal using the fact that $\scS_0$ is a (Abelian) group on $X_{ap}=VX$ and $P^{\sun}$ is a minimal chosen idempotent. 
Moreover, using $\scS_0$ as Abelian, we find that $VP$ is an idempotent with $V(VP)=VP$; hence, $VP=V.$ 
Similarly, we obtain from $P^{\sun}(P^{\sun}V^{\sun})=P^{\sun}V^{\sun};$ hence, $P^{\sun}=P^{\sun}V^{\sun}$ because of its minimality.
This result leads to
\begin{eqnarray*}
<x,\eta(V)x^{\sun}>&=&<x,\eta(V\circ P)x^{\sun}>=<x,P^{\sun}V^{\sun}x^{\sun}>\\
&=&<x,P^{\sun}x^{\sun}> =<x,\eta(P)x^{\sun}>.
\end{eqnarray*}
In the first line, $V$ left minimal is used, and in the
second, $P^{\sun}$ left minimal is used.
 Because $\eta$ is injective, we have that $V=P$ and $\scS_0(X)\subset X$ by the Eberlein weakly almost periodicity; we conclude that $P(X)=V(X)\subset X.$
\end{proof}

\section{The restriction semigroup}
The main topic of this section is to show how close the space of almost periodic vectors are to the reversible ones. 
It is shown that separability of the orbit is the key indicator that they coincide. In short, the separability of the orbit of a vector $x$ mainly serves to obtain an almost periodicity property of the splitting. 
In consequence, the above lead to an extension of the well-known result on the ergodic properties of a dual semigroup.

Every semigroup  $\bk{T(t)}_{t\ge 0}$ is the restriction of the dual semigroup $\bk{T^{\sun\sun}(t)}_{t\ge 0}.$ In this chapter, we want to show how the theory developed for dual semigroups applies to general $C_0-$semigroups, which are restrictions to norm closed subspaces $Y\subset X^{\sun}.$

\begin{defi}
A vector $x^{\sun}\in X^{\sun}$ is an eigenvector with unimodular eigenvalue if for a map 
$\funk{\la}{\scS_0^{\sun}}{\ce}$ with $\btr{\la(T)}=1$, we have $Tx^{\sun}=\la(T)x^{\sun}$ for all $T\in \scS_0^{\sun}$. We define
$$X_{uds}^{\sun}:=\overline{span}\bk{x^{\sun}\in X^{\sun}: x^{\sun} \mbox{ is an eigenvector with unimodular eigenvalue }}.
$$
\end{defi}

\begin{pro} Let $\scS_0^{\sun}$ be Abelian; then, $X^{\sun}_{uds}\subset X^{\sun}_a$.
\end{pro}
\begin{proof}
For a given $x^{\sun}\in X^{\sun}_{uds},$ the minimal idempotent $P^{\sun}$, and all $T\in\scS_0^{\sun},$ we have $Tx^{\sun}=\la(T)x^{\sun}.$ Hence, 
$$ 
\la(T)P^{\sun}x^{\sun}=P^{\sun} \la()x^{\sun}=P^{\sun}Tx^{\sun}=TP^{\sun}x^{\sun}.
$$
The splitting with respect to $P^{\sun}$ gives $x_0^{\sun}=x^{\sun}-x_a^{\sun},$  and by the previous observation, $T^{\sun}x_0=\la(T^{\sun})x_0.$  Because $x_0$ is a flight vector, we find a net with
$$0=\netlim{\la}{\Lambda}<x,T^{\sun}(t_{\la})x_0^{\sun}>=\netlim{\la}{\Lambda}\la(T^{\sun}(t_{\la}))<x,x^{\sun}_a-x^{\sun}>.
$$
Because $\btr{\la(T^{\sun}(t_{\la}))}=1$, we conclude that $x^{\sun}_a=x^{\sun}.$
\end{proof}
In the Abelian case, we recall that, by  \cite[Thm. 4.1 p. 104]{Krengel}, the minimal ideal is unique. Moreover, we know from Theorem 
\ref{right_topo_semigroups} that $K:=P^{\sun}\scS_0^{\sun}P^{\sun}=P^{\sun}{\scS_0^{\sun}}_{|X^{\sun}_a}$ is a group. 

Next, we discuss the weak-star topology on $Y,$ that is, $\overline{O(x)}^{w*}\subset Y$ for all $x\in Y.$ Then, we consider the mapping
$$
\Funk{r}{\scS^{\sun}_0}{L(Y),}{S}{S_{|Y}.}
$$
Defining $\scR^{\sun}=r(\scS^{\sun})$ and $\scR_0^{\sun}:=r(\scS_0^{\sun}),$ we equip $\scR_0^{\sun}$ with the topology $\scr,$  that is,
$\net{R}{\la}{\Lambda}$ is convergent if $<R_{\la}y,x>\to <Ry,x>$ for all $y\in Y $ and $x\in X.$

\begin{defi} 
Let $\tau$ be a locally convex topology on $X^{\sun}$, with  $\sigma(X^{\sun},X)\subset \tau\subset \nrm{\cdot}.$ We call a net $\net{T}{\al}{A}$ $\tau-OT$ convergent if there exist a $T\in L(Y)$ with
$$\tau-\netlim{\al}{A}T_{\la}y=Ty \mbox{ for all } y\in Y.
$$
\end{defi}

\begin{pro} \label{res-l-R-continuity}
Let $V\in L(Y),$ $t\in\re.$  Then,
$$
\Funk{L}{(L(Y),\tau-OT)}{(L(Y),\tau-OT)}{W}{W V}
$$
is continuous. Moreover, for $t\in \rep$ and $S(t)=r(T^{\sun}(t)),$
$$
\Funk{R}{(L(Y),\scr)}{(L(Y),\scr)}{W}{S(t)W}
$$
is continuous
\end{pro}
\begin{proof}
Let $\net{W}{\al}{A}\subset L(Y)$ $\tau-OT$ be convergent with the limit $W.$ Then, due to $Vy\in Y,$
$$
\tau\netlim{\al}{A}W_{\al} Vy = WVy \mbox{ for all } y\in Y, \ x\in X,
$$
which gives that $R$ is continuous.
Let $t\in\rep,$ as $S(t)=r(T^{\sun}(t)), $ and let $\net{W}{\al}{A}\subset L(Y)$ $\scr$ be convergent with the limit $W.$ We have for $y\in Y$ and $x\in X$ that
$$
<S(t)W_{\al}y,x>=<T^{\sun}(t)W_{\al}y,x>=<W_{\al}y,T(t)x>\to <Wy,T(t)x> \mbox{ for all } y\in Y, \ x\in X.
$$
\end{proof}

Next, we verify some basic properties for the defined set $\scR_0^{\sun}.$
\begin{pro} \label{restriction-basics}
Let $Y\subset X^{\sun}$ such that $\overline{O(y)}^{w*}\subset Y$ for all $y\in Y$  and  $P^{\sun} \in \scS_0^{\sun}$ be a minimal idempotent. Then, 
\begin{enumerate}
\item $S(t)=r(T^{\sun}(t))$.
\item $(L(Y),\scr)$ is a Hausdorff locally convex space.
\item $\funk{r}{(\scS_0^{\sun},\kappa^{\sun})}{(\scR_0^{\sun},\scr)}$ is continuous.
\item \label{restricted-homorphism} $r(ST)=r(S)r(T)$ for all $S,T\in \scS_0^{\sun}$.
\item \label{restricted-semitopological} $(\scR_0^{\sun},\scr)$ is a left topological semigroup.
\item \label{restricted-idempotent} $Q=r(P^{\sun})$ is a projection onto a $Y_a$.
\item $r(P^{\sun}\scS_0^{\sun}P^{\sun})$ is a compact topological group with the identity $Q:=r(P^{\sun}).$ Moreover, $r(P^{\sun}\scS_0^{\sun})$ is a group on $Y_a.$
\item If the function $ \bk{t\mapsto <T^{\sun}(t)y,x>}\in W(\rep)$ for all $y\in Y$ and $x\in X,$ then $\scR_0^{\sun}$ is Abelian.
\end{enumerate}
\end{pro}
\begin{proof}
Let $T\in L(Y)$ and $<Ty,x>=0$ for all $y\in Y, \ x\in X$. Since $Ty\in X^{\sun}\subset X^*$ may be viewed as a linear functional, $\funk{Ty}{X}{\ce}.$ Hence, $Ty=0 $ for all $y\in Y,$ which gives $T=0.$

As the topology $\scr$ is weaker than $\kappa^{\sun}$, the first claim is verified. Using $Y\subset X^{\sun}$ such that $\overline{O(x)}^{w*}\subset Y$ for all $y\in Y,$ we have $Ty\in Y$ for all $y\in Y,$ and $T\in \scS_0^{\sun}.$ Hence, $TS_{|Y}=T_{|Y}S_{|Y}.$ 

That $\scR_0^{\sun}$ is a left semitopological semigroup is a consequence of Proposition \ref{res-l-R-continuity}.

Let $y\in Y_a\subset X_a,$ which gives $y=P^{\sun}y=P^{\sun}_{|Y}y=Qy.$ 

Let $V\in r(P^{\sun}\scS_0^{\sun}P^{\sun})$. Then, there is a $T\in K:=P^{\sun}\scS_0^{\sun}P^{\sun}$, with $V=r(T).$ As $K$ is a group, we find $S\in K$, with $TS=P^{\sun},$ and now apply (\ref{restricted-homorphism}) and (\ref{restricted-idempotent}). The additional claim is a consequence of (\ref{restricted-idempotent}).

To verify the compactness, note that $\scr$ is weaker than $\kappa^{\sun}.$ 

The last claim is a consequence of Theorem \ref{pointwise_Abelian}.
\end{proof}

After this preparation, the situation above can be generalized to a locally convex topology $\tau$ defined on $X^{\sun}$ such that for $Y\subset X^{\sun}$, the following properties hold:
\begin{enumerate}
\item \label{stronger_top} $\sigma(X^{\sun},X)\subset \tau \subset \nrm{\cdot}$.
\item $\overline{O(y)}^{\tau} \subset Y$ for all $y\in Y$.
\item \label{rudin_integral} $\overline{ac}^{\tau}O(y)$ is $\tau$ compact for all $y\in Y.$
\end{enumerate}
Note that by (\ref{stronger_top}) and (\ref{rudin_integral}), 
$$
\scR_0^{\sun}:=\bk{U:Uy=\tau-\netlim{\lambda}{\Lambda}T^{\sun}(t_{\lambda})y, \mbox{ for all } y\in Y, \ \net{t}{\lambda}{\Lambda}\subset \rep}.
$$

Applying Tychnov's Theorem, we have that
$$
\scR^{\sun}\subset ac\scR^{\sun}\subset \Pi_{y\in Y} (\overline{ac}^{\tau}{O(y)})
$$
are relatively compact with respect to the operator topology $\tau-OT.$ 
We have
$$\scR_0^{\sun}=\overline{\scR^{\sun}}^{\tau-OT}$$
and define 
$$
\scW_0^{\sun}:=\overline{ac}^{\tau-OT}\bk{S: S\in ac\scR_0^{\sun}};
$$
we have that they are compact with respect to $\tau-OT.$ 
The $\tau$ compactness and (\ref{stronger_top}) give that $\tau-OT$ and 
$\scr$ are equal on $\scR_0^{\sun}$ and $\scW_0^{\sun}.$
Thus, Proposition \ref{restriction-basics} becomes true if $\scr$ is replaced by the topology $\tau-OT.$ Consequently, the mappings
$$
\Funk{R}{(\scR_0^{\sun},\tau-OT)}{(\scR_0^{\sun},\tau-OT)}{W}{S(t)W}
$$
and, for all $y\in Y,$
$$
\Funk{\delta_y}{(\scR_0^{\sun},\tau-OT)}{(Y,\tau)}{S}{Sy}
$$
are continuous.  

\begin{pro} \label{restricted-semitop-convex-cpt}
If 
\begin{enumerate}
\item $\sigma(X^{\sun},X)\subset \tau\subset \nrm{\cdot}$
\item $\overline{O(y)}^{\tau} \subset Y$ for all $y\in Y$, and
\item $\overline{ac}^{\tau}O(y)$ is $\tau$ compact for all $y\in Y$
\end{enumerate}
hold, then
$\scW_0^{\sun}$ is a compact left semitopological semigroup, and if $\scR_0^{\sun}$ is Abelian, then $\scW_0^{\sun}$ is as well. 
\end{pro} 
\begin{proof} The compactness was proved in the previous remarks; hence, we may only consider the topology $\scr$ on $\scW_0^{\sun}.$
Now, let $U,V\in\scW_0^{\sun}.$ Then, we find nets 
$\bk{t_i^{\gamma}}_{i\in\za,\gamma\in\Gamma},\bk{s_i^{\lambda}}_{i\in\za,\lambda\in\Lambda}\subset\rep$ 
and 
$$\bk{\alpha_i^{\gamma}}_{i\in\za,\gamma\in\Gamma}, \bk{\beta_i^{\lambda}}_{i\in\za,\lambda\in\Lambda}\subset \re,$$
with  
$\sum_{i=1}^{n_{\gamma}} \btr{\alpha_i^{\gamma}}\le1$
and 
$\sum_{i=1}^{m_{\lambda}}\btr{\beta_i^{\lambda}}\le1,$ such that

$$
U_{\gamma}:=\sum_{i=1}^{n_{\gamma}}\alpha_i^{\gamma}T^{\sun}(t_i^{\gamma})\mbox{ with } \netlim{\gamma}{\Gamma}U_{\gamma}=U \mbox{ and }
V_{\lambda}=\sum_{i=1}^{m_{\lambda}}\beta_i^{\lambda}T^{\sun}(s_i^{\lambda}), \mbox{ with } \netlim{\lambda}{\Lambda}V_{\lambda}=V.
$$
First, note that
\begin{equation} \label{VU}
VU=\netlim{\lambda}{\Lambda}\netlim{\gamma}{\Gamma}
\sum_{i=1}^{n_{\gamma}}\sum_{j=1}^{m_{\lambda}}\alpha_i^{\gamma}\beta_j^{\lambda}T^{\sun}(t_i^{\gamma}+s_{j}^{\lambda}),
\end{equation}
with a right hand side in $\scW_0^{\sun}.$ Thus, the compactness completes the proof.
 
For the second claim, define $g(t):=<x,T^{\sun}(t)x^{\sun}>,$ which is assumed to be Eberlein weakly almost periodic, for $x\in X$ and $x^{\sun}\in Y,$ due to $\scR_0^{\sun}$ being Abelian. Further, let $\delta_t(g):=g(t)$ be the bounded linear functional on $BUC(\rep,X).$ 

Recalling (\ref{VU}), we have
$$
UV=\netlim{\gamma}{\Gamma}\netlim{\lambda}{\Lambda}
\sum_{i=1}^{n_{\gamma}}\sum_{j=1}^{m_{\lambda}}\alpha_i^{\gamma}\beta_j^{\lambda}T^{\sun}(t_i^{\gamma}+s_{j}^{\lambda}).
$$
In consequence, we only have to verify that we can interchange the limits. The duality reads as
\begin{eqnarray*}
x(VUx^{\sun})&=&\netlim{\lambda}{\Lambda}<x,V_{\lambda}Ux^{\sun}>\\
&=&\netlim{\lambda}{\Lambda}\netlim{\gamma}{\Gamma}
\sum_{i=1}^{n_{\gamma}}\sum_{j=1}^{m_{\lambda}}\alpha_i^{\gamma}\beta_j^{\lambda}<x,T^{\sun}(t_i^{\gamma}+s_{j}^{\lambda})x^{\sun}>\\
&=&\netlim{\lambda}{\Lambda}\netlim{\gamma}{\Gamma}
\sum_{i=1}^{n_{\gamma}}\sum_{j=1}^{m_{\lambda}}\alpha_i^{\gamma}\beta_j^{\lambda}g(t_i^{\gamma}+s_{j}^{\lambda}) \\
&=&\netlim{\lambda}{\Lambda}\netlim{\gamma}{\Gamma}
<\sum_{i=1}^{n_{\gamma}}\alpha_i^{\gamma}g(\cdot+t_i^{\gamma}),\sum_{j=1}^{m_{\lambda}}\beta_j^{\lambda}\delta_{s_j^{\lambda}}>_{(BUC(\rep),BUC(\rep)^*)}\\
&=&\netlim{\gamma}{\Gamma}\netlim{\lambda}{\Lambda}
<\sum_{i=1}^{n_{\gamma}}\alpha_i^{\gamma}g(\cdot+t_i^{\gamma}),\sum_{j=1}^{m_{\lambda}}\beta_j^{\lambda}\delta_{s_j^{\lambda}}>_{(BUC(\rep),BUC(\rep)^*)}.
\end{eqnarray*}
As $ac O(g)$ is weakly relatively compact in $BUC(\rep)$, we may interchange the limits.
\end{proof}

Next, we recall the consequence of \cite[Cor. 2 (a), p. 127]{Jarchow}, and we have the following.
\begin{pro} \label{tau-separated}
Let $E\subset F\subset X^*$ and $\tau$ be a locally convex topology on $X^*$, with $\sigma(X^*,X)\subset  \tau \subset \nrm{\cdot}.$  
If $E$ cannot be separated from $F$ by a $\tau-$continuous functional, then 
$\overline{E}^{\tau}=\overline{F}^{\tau}.$
\end{pro}
\begin{proof}
If $\overline{E}^{\tau}\not =\overline{F}^{\tau},$ then there exist an $x\in \overline{F}^{\tau} \backslash \overline{E}^{\tau}$ and a $\tau-$continuous functional $\phi$ such that $\phi_{|\overline{E}^{\tau}}=0$ and $\phi(x)=1.$ By definition, we have for net $\net{x}{\la}{\Lambda}\subset F$ the $\tau$ convergence $x_{\la}\to x.$ Moreover, we find a subnet that has no intersection with $\overline{E}^{\tau}.$ The continuity $\phi$ leads to an element $x_{\la_0},$ with $\phi(x_{\la_0})>1/2,$ which gives the contradiction.
\end{proof}

\begin{lem} \label{res-Fourier_integral}
Let $\scR_0^{\sun}$ be a restricted semigroup and $P^{\sun}$ be the minimal idempotent of $\scS_0^{\sun};$ additionally, let
\begin{enumerate}
\item  $\sigma(X^{\sun},X)\subset \tau\subset \nrm{\cdot}$,
\item $\overline{O(y)}^{\tau} \subset Y$ for all $y\in Y$,
\item  $\overline{ac}^{\tau}O(y)$ be $\tau$ compact for all $y\in Y$, and
\item $\bk{\rep\ni t \mapsto <T^{\sun}(t)y,x>}\in W(\rep)$ for all $y\in Y, x\in X$.
\end{enumerate}
Then, 
 $G:=r(P^{\sun}{\scS_0^{\sun}})_{|Y_a},$ is a compact Abelian topological group. Further, let $\Gamma$ be the character group of $G$, $\gamma\in\Gamma,$ and $\rho$ denote the normalized Haar measure on $G.$  Then,
$$
S_{\gamma}:=\int_G\overline{\gamma}(S)Sd\rho(S)
$$
exists in the sense of \cite[Def 3.26, p. 74]{RudinFA} in $(L(Y),\tau-OT)$, and 
$
S_{\gamma}\in \scW_0^{\sun} .
$ 
For given $V\in \scR_0^{\sun}$, we have
$$
S_{\gamma}Vx=\int_G\overline{\gamma}(S)S Vx d\rho(S)
$$
in $(L(Y),\tau-OT).$
\end{lem} 

\begin{proof}
We start with verifying the closedness of $G$ as a subset of $\scR_0^{\sun}.$ Therefore, let $\net{T}{\la}{\Lambda}\subset\scR_0^{\sun}$ such that for $Q=r(P^{\sun})$,
$$
S=\tau-\netlim{\la}{\Lambda}r(P^{\sun}T_{\la})=\tau-\netlim{\la}{\Lambda}Qr(T_{\la}).
$$
By the compactness of $\scS_0^{\sun}$, we may assume that $T=\netlim{\la}{\Lambda}T_{\la}.$
By the above, for given $x^{\sun}\in Y$ and $x\in X$, we have
\begin{eqnarray*}
Sx^{\sun}&=&\tau-\netlim{\la}{\Lambda}Qr(T_{\la})x^{\sun} \\
&=&\tau-\netlim{\la}{\Lambda}r(T_{\la})Qx^{\sun} \\
&=&r(T)Qx^{\sun} \\
&=&Qr(T)x^{\sun},
\end{eqnarray*}
which verifies the closedness. As $G$ is Abelian from Proposition \ref{restriction-basics} (5), we find that $G$ is semitopological, and from abstract harmonic analysis \cite{Ellis_cont}, we recall that any compact semitopological group is a topological group. Hence, we find the normalized Haar measure $\rho$ on $G,$ \cite[Thm 5.14, p. 123]{RudinFA}.

To prove the existence of the integral, we apply Theorem \cite[Thm. 3.27, pp. 74-75]{RudinFA}. By Proposition \ref{restriction-basics} (2), (a) of the cited Thm 3.27 is verified.
By \cite[Thm. 5.14 p. 123]{RudinFA} $(G,\rho)$ is a Borel probability measure space.
Because, $G\subset L(Y)$ and carries the topology $\tau-OT,$
$\funk{id}{G}{(L(Y),\tau-OT)}$ 
is continuous; therefore,
$$
\Funk{f}{G}{(L(Y),\tau-OT)}{S}{\overline{\gamma}(S)S}
$$ 
as well. Moreover, 
$$
f(G)=\bk{\overline{\gamma}(S)S:S\in G}\subset \scW_0^{\sun}\subset L(Y),
$$
with $\scW_0^{\sun}$ convex and $\tau$ compact. Consequently, the integral exists and is an element of $\scW_0^{\sun}.$ For the additional proof, note that, for $x^{\sun}\in Y,$
$$
\Funk{\delta_{x^{\sun}}}{(L(Y),\tau-OT)}{(Y,\tau)}{S}{Sx^{\sun}}
$$
and, for $V\in L(Y)$,
$$
\Funk{L}{(L(Y),\tau-OT)}{(L(Y),\tau-OT)}{W}{WV}
$$
are continuous linear operators, and the claims become a consequence of \cite[p.85, Exercise 24]{RudinFA}.

\end{proof}

Now, we are ready to present the first main result of this section by showing that the space of unimodular eigenvectors is almost $Y_a.$

\begin{theo} \label{restricted_Xuds=X_a}
Let $Y\subset X^{\sun}$ be a norm closed subspace, $\scR_0^{\sun}$ be a restricted semigroup, and $P^{\sun}$ be the minimal idempotent of $\scS_0^{\sun};$ additionally, let
\begin{enumerate}
\item $\tau$ be a locally convex topology on $X^{\sun}$,
\item  $\sigma(X^{\sun},X)\subset \tau\subset \nrm{\cdot}$,
\item $\overline{O(y)}^{\tau} \subset Y$ for all $y\in Y$,
\item  $\overline{ac}^{\tau}O(y)$ be $\tau$ compact for all $y\in Y$, and
\item $\bk{\rep\ni t \mapsto <T^{\sun}(t)y,x>}\in W(\rep)$ for all $y\in Y, x\in X.$
\end{enumerate}
Then, 
$$
\overline{Y_a}^{\tau}\subset \overline{X_{uds}^{\sun}\cap Y}^{\tau}.
$$
\end{theo}

\begin{proof}
\underline{Part 1:} In this part, we prove that unimodular eigenvectors are found as an image of the generalized Fourier transforms for a given $\gamma\in \Gamma.$
By Lemma \ref{res-Fourier_integral}, for the minimal idempotent $P^{\sun}$ and $Q=r(P^{\sun})$, we find  that $\rho$ is the normalized Haar measure on the Abelian compact topological group 
$G=r(P^{\sun}\scS_0^{\sun}).$  
Further, if $\Gamma$ denotes the character group, for $\gamma\in\Gamma$, we can define 
$$
S_{\gamma}:=\int_G\overline{\gamma}(S)Sd\rho(S) \in\scW_0^{\sun}\subset L(Y).
$$
Consequently, for $x^{\sun}\in Y$, we have $S_{\gamma}x^{\sun}\in Y,$ and because $S_{\gamma}\in \scW_0^{\sun}$ and $\scW_0^{\sun}$ is Abelian by Proposition \ref{restricted-semitop-convex-cpt}, we find that $S_{\gamma}$ commutes with the operators in $\scR_0^{\sun}\subset \scW_0^{\sun}.$ Using Lemma \ref{res-Fourier_integral} for $R\in G, x^{\sun}\in Y$, we find that
\begin{eqnarray*}
RS_{\gamma}x^{\sun}&=&S_{\gamma}Rx^{\sun}=\delta_{x^{\sun}}(S_{\gamma}^{\sun}R)\\
&=&\delta_{x^{\sun}}\fk{\int_G \overline{\gamma}(S)SRd\rho(S)} \\
&=&\int_G \overline{\gamma}(S)SRx^{\sun}d\rho(S) \\
&=&\int_G \overline{\gamma}(S)RSx^{\sun}d\rho(S) \\
&=&\gamma(R)\int_G \overline{\gamma}(RS)RSx^{\sun}d\rho(S) \\
&=&\gamma(R)\int_G \overline{\gamma}(S)Sx^{\sun}d\rho(S) 
\mbox {  apply \cite[Thm 5.14 (1),(2), p. 123]{RudinFA}}\\
&=&\gamma(R)S_{\gamma}x^{\sun}.
\end{eqnarray*}
Similarly, using the fact that $G$ is Abelian, we obtain
\begin{equation}
RS_{\gamma}^{\sun}=\gamma(R)S_{\gamma}^{\sun}=S_{\gamma}^{\sun}R.
\end{equation}
Because $Q$ is the unit in $G,$ we have $\gamma(Q)=1,$ and by the previous observation, 
$  P^{\sun}S_{\gamma}=S_{\gamma}.$ Hence, for $T\in \scR_0^{\sun},$ we find $QT\in G$ and
$$
TS_{\gamma}^{\sun}=TQS_{\gamma}^{\sun}=\gamma(TQ)S_{\gamma}^{\sun}=\gamma(T)S_{\gamma}^{\sun}.
$$
This means that $S_{\gamma}Y$ consists of eigenvectors with unimodular eigenvalues $\la(T)=\gamma(T).$

\underline{Part 2:} Therefore, let $\Gamma$ be the character group of $G=r(P^{\sun}\scS_0^{\sun})_{|Y_a}.$

We prove that $Y_{a}$ cannot be separated from
$$
M=\overline{span}\bk{S_{\gamma}x^{\sun}:y \in Y, \ \gamma\in \Gamma}
$$
with a $\tau-$continuous functional $\phi$ and apply Proposition \ref{tau-separated}

Because $M\subset X_{uds}^{\sun}\subset X_{a}^{\sun},$ we assume that there is a $y\in Y_{a} \backslash M.$  By the assumption, we will find a $\tau$ continuous $\phi$ such that for $Q=r(P^{\sun}),$ $\phi(Qy)=\phi(y) \not= 0$ and $\phi_{|M}=0.$ Using the fact that $\Lambda:=\bk{L(Y)\ni T\mapsto \phi(Tx)}$ is $\tau-OT$ continuous, we obtain
 
\begin{equation} \label{G_integrals_eq_null}
0=\phi(S_{\gamma}z)=\int_G\overline{\gamma}(S)\phi(Sz)d\rho(S)
\end{equation}
for all $\gamma\in \Gamma$ and $z\in Y.$

Because the characters form an orthonormal basis in $L_2(G,\rho)$, see \cite[p. 944]{DS}, we have 
$$\bk{G\ni S\mapsto \phi(Sy)} =0  \ a.e.$$ 

Using the fact that $G$ carries the topology $\tau,$ we have for $\phi$ $\tau$ continuous and $ z\in Y$ that the functions
 
$$
\Funk{g}{(G,\tau-OT)}{\ce}{S}{\phi(Sz)}
$$

are continuous. Consequently,$ \bk{G\ni S\mapsto \phi(Sy)}$  is identically zero, and we find a contradiction to $\phi(Qy)\not=0,$ which completes the proof.
\end{proof}

As a direct consequence, we obtain the following for bounded sun-dual-semigroups.
\begin{theo} \label{reccurent_equals_unimodular}
 If $\scS_0^{\sun}$ is Abelian, then
$$
\overline{X^{\sun}_{uds}}^{\sigma(X^{\sun},X)}=\overline{ X_a^{\sun}}^{\sigma(X^{\sun},X)}.$$
\end{theo}

The result of Frechet for asymptotically almost periodic functions becomes a direct consequence.
\begin{theo}
Let $Z$ be a Banach space and $O(x)$ for all $x\in Z$ be relatively compact; then,
$Z_{uds}=Z_a.$
\end{theo}
\begin{proof}
Choose $\tau$ as the norm topology, $Y:=Z$ and $X^{\sun}=Z^{\sun\sun}.$ The additional proof becomes straightforward. 
\end{proof}

Now, we are ready to show how a spectral condition applies to obtain a zero mean.
\begin{cor}
Let $A$ be the generator of $\scS$ and $\sigma_p(A^{\sun})\cap i\re=\emptyset.$ If 
$\bk{t\mapsto <x,T^{\sun}(t)x^{\sun}>}$ is Eberlein weakly almost periodic for all $x\in X, x^{\sun}\in X^{\sun},$ then
\begin{equation} \label{ergdic-0}
\lim_{T\to \infty}\frac{1}{T}\int_0^T\btr{<x,T^{\sun}(t)x^{\sun}>}dt=0, \ \mbox{ for all } x\in X, \  x^{\sun}\in X^{\sun}.
\end{equation}
\end{cor}

\begin{proof}
Let (\ref{ergdic-0}) be not equal to zero for some pair $(x,x^{\sun})$; then, by \cite[Thm. 4.7, p. 108]{Krengel}, we find that
$\bk{t\mapsto <T^{\sun}(t)x^{\sun},x>}\not\in W_0(\rep).$
 Consequently, $x^{\sun}$ is not a $\kappa^{\sun}-$flight vector. By the splitting obtained in Theorem \ref{apply-to-sun}, $x^{\sun}=x^{\sun}_a\oplus x_0^{\sun},$ the Abelian structure gives $X_{fl}^{\sun}=X_0^{\sun}$ and therefore  
 an $x_a^{\sun}\not= 0.$ By Theorem \ref{reccurent_equals_unimodular}, $X_{uds}^{\sun}\not= \bk{0},$ 
 which leads to a unimodular eigenvector of $\bk{T^{\sun}(t)}_{t\in\rep}$ and therefore to an element of $\sigma_p(A^{\sun})\cap i\re.$
\end{proof}

\begin{remk}
Deleeuw-Glicksberg needs $\bk{t\mapsto T^{\sun}(t)x^{\sun}}$  Eberlein weakly almost periodicity for all $x\in X^{\sun}$. In Corollary 7.12 above, only $\bk{t\mapsto <y,T^{\sun}(t)x^{\sun}>}$ needs Eberlein weakly almost periodicity for all $x^{\sun}\in X^{\sun},y\in X.$
\end{remk}

Due to the pointwise verification of the Abelian structure, we can give the following criterion for a vector to be a member of $X_{uds}^{\sun},$ which is the second main result of this section. 
\begin{theo}
Let $Y\subset X^{\sun}$ be a norm closed subspace, $\scR_0^{\sun}$ be a restricted semigroup, and $P^{\sun}$ be the minimal idempotent of $\scS_0^{\sun};$ additionally, let
\begin{enumerate}
\item $\tau$ be a locally convex topology on $X^{\sun}$,
\item  $\sigma(X^{\sun},X)\subset \tau\subset \nrm{\cdot}$,
\item $\overline{O(y)}^{\tau} \subset Y$ for all $y\in Y$,
\item  $\overline{ac}^{\tau}O(y)$ be $\tau$ compact for all $y\in Y$, and
\item $\bk{\rep\ni t \mapsto <T^{\sun}(t)y,x>}\in W(\rep)$ for all $y\in Y, x\in X.$
\end{enumerate}
Then, the following for $x^{\sun}\in Y_a$ are equivalent:
\begin{enumerate}
\item $\overline{O(x^{\sun})}^{\tau}$ is norm separable.
\item $x^{\sun}\in X_{uds}^{\sun}$.
\item $O(x^{\sun})$ is relatively norm compact.
\end{enumerate}
\end{theo}

\begin{proof}
As ${\scR_0^{\sun}}$ is Abelian, $G=r(P^{\sun}\scS_0^{\sun}P^{\sun})$ is a compact Abelian topological group \cite{Ellis}. 
The splitting is a consequence of Thm \ref{apply-to-sun} (\ref{sun-subspace-split}). From Corollary \ref{res-Fourier_integral}, we have
$$
S_{\gamma}x^{\sun}=\int_{G}\overline{\gamma}(S)Sx^{\sun} d\rho(S)\in Y.
$$
By \cite[Cor. 4 pp. 42-43]{DiestelUhl}, we have its a Bochner integral, which is an element of $X^*.$ Moreover, it coincides on $X$ with the integral defined in the proof of theorem \ref{reccurent_equals_unimodular}; hence, it becomes an element of $Y.$

For $R\in \scR_0^{\sun}$, we use that the left multiplication is continuous and find that  

\begin{eqnarray*}
RS_{\gamma}y&=&\int_G\overline{\gamma}(S)SRy d\rho(S) =\int_{G}\overline{\gamma}(S)RSy d\rho(S)=\gamma(R)S_{\gamma}y.
\end{eqnarray*}

Defining $$M=\overline{span}\bk{S_{\gamma}y:y\in Y, \gamma\in \Gamma},$$ we have $M\subset Y_a$ .

For $x^{\sun}\in X_a^{\sun}$ and $q:X^{\sun}_a\to X_a^{\sun}/M$ as the quotient map, if $Z=\overline{span}\bk{qGx^{\sun}}$, then by assumption, $(Z,\nrm{\cdot})$ is separable. 
Consequently, $(B_{Z^*}, w^*)$ is separable (compact metrizable). Choose $\bk{z^*_n}_{n\in\za}$ dense in $(B_{Z^*}, w^*),$  and use the Hahn-Banach-Theorem to extend them norm-preservingly to all of $X_a^{\sun}/M.$

By definition, $S_{\gamma}x^{\sun}\in M.$ Consequently, for the sequence of bounded linear functionals 
$$
\Funk{U_n}{X^{\sun}}{\ce}{u}{<qu,z_n^*>,} 
$$
due to the Bochner integrability, we obtain,
$$
0=<qS_{\gamma}x^{\sun},z_n^*>=\int_G\overline{\gamma}(S)<qSx^{\sun},z_n^*>d\rho(S)
$$
for all $\gamma\in \Gamma$ and $n\in \za.$ Using $\bk{\gamma}_{\gamma\in \Gamma}$ as an orthonormal basis in $L^2(G,\rho)$,
$$
<qSx^{\sun},z_n^*>=0 \mbox{ a.e. for all } n\in \za.
$$
Thus, for sets $A_n\subset K,$ with $\rho(A_n)=0$, we have
$$
<qSx^{\sun},z_n^*>=0 \mbox{ for all } S\in G\backslash A_n, n\in \za.
$$
Let $A=\bigcup_{n\in\za}A_n;$ then, $\rho(A)=0$ and
$$
<qSx^{\sun},z_n^*>=0 \mbox{ for all }S\in G\backslash A, n\in \za.
$$
Using $\bk{z^*_n}_{n\in\za}$ totally on $Z$, we find an $S\in G$ with $qSx^{\sun}=0. $ Consequently, $Sx^{\sun}\in M.$ Because of Part 1 of Theorem \ref{reccurent_equals_unimodular}, the space $M$ is translation invariant, and for $x^{\sun}\in X_a$, we find, using $G$ as a group on $X_a,$ a $T\in G$ such that $TSx^{\sun}=x^{\sun}$ and therefore $x^{\sun}\in M\subset X_{uds}^{\sun}.$

(2)$\Rightarrow$ (3): Let $x^{\sun} \in X^{\sun}_{uds}$ Then, $x^{\sun}$ is the limit of linear combinations of unimodular vectors $\bk{x_i^{n}}_{i=1..m_n,n\in\za}\subset X_a^{\sun}$, i.e., satisfying $Tx_i^n=\la^n_i(T)x_i^n.$  Consequently, $O(x_i^n)$ is norm compact and therefore the orbit of the linear combination. 
This leaves, if the vectors $\bk{x_n}_{n\in\za}$ have relatively norm-compact orbits and $x_n \to x,$ that $O(x)$ is relatively norm compact. Note that for some constant $C>0$,
$$
\nrm{Tx_n-Tx}\le C\nrm{x-x_n},
$$
which concludes the proof
\end{proof}

As proposed, we show that the separability of the orbit indicates almost periodicity. Note that if $x^{\sun}\in X_{uds}^{\sun},$ then the mapping $\bk{t\mapsto T^{\sun}(t)x^{\sun}}$ is almost periodic. Therefore, we give a criterion when an element in $X_a^{\sun}$ is in $X_{uds}^{\sun}.$ 
\begin{theo}\label{separable-theorem}
If $x^{\sun}\in X_a^{\sun}$ and  $\scS_0^{\sun}$ is Abelian, then the following are equivalent.
\begin{enumerate}
\item $\overline{O(x^{\sun})}^{\sigma(X^{\sun},X)}$ is norm separable.
\item $x^{\sun}\in X_{uds}^{\sun}$.
\item $O(x^{\sun})$ is relatively norm compact.
\end{enumerate}
\end{theo}

\begin{cor}
If $\overline{O(x^{\sun})}^{\sigma(X^{\sun},X)}$ is norm separable for all $x^{\sun}\in X_a^{\sun}$ and  $\scS_0^{\sun}$ is Abelian, then
$$X_a^{\sun}= X_{uds}^{\sun}.$$
\end{cor}

\begin{cor}
If $X^{\sun}$ is norm separable and  $\scS_0^{\sun}$ is Abelian, then
$$
X^{\sun}_{a}=X^{\sun}_{uds}.
$$
\end{cor}

\begin{pro} Let $\bk{S(t)}_{t\ge0}$  and $x\in X;$ then,
\begin{enumerate}
\item If $\overline{O(x)}$ is compact, then $\overline{O(x)}$ is norm separable and $\overline{\bk{S^{\sun\sun}(t)x:t\ge 0}}^{\sigma(X^{\sun\sun},X^{\sun})}\subset X.$
\item If $\overline{O(x)}$ is weakly compact, then $\overline{O(x)}^{\sigma(X^{\sun\sun},X^*)}$ is norm separable and  
$$\overline{\bk{S^{\sun\sun}(t)x:t\ge 0}}^{\sigma(X^{\sun\sun},X^{\sun})}\subset X.$$
\end{enumerate}
\end{pro}

\begin{proof} We only prove the second claim.
Because $\bk{t\mapsto S(t)x}$ is weakly continuous, $\overline{O(x)}^{\sigma(X,X^*)}$ is weakly separable, hence the closed convex hull and therefore norm separability. 
As 
\begin{eqnarray*}
<T^{\sun\sun}(t)jx,x^*>&=&<jx,T^{\sun}(t)x^{\sun}>=<x,T^{\sun}(t)x^{\sun}>\\
&=& <T(t)x,x^{\sun}>=<jT(t)x,x^{\sun}>,
\end{eqnarray*}
$ jT(t)x=T^{\sun\sun}(t)jx$, and because $X^{\sun}\subset X^*$, the weak compactness gives
$$
\overline{O(jx)}^{\sigma(X^{\sun\sun},X^{\sun})}=j\overline{O(x)}^{\sigma(X,X^{\sun})}\subset jX.
$$
\end{proof}

\begin{remk} If $x^{\sun}=x_a^{\sun}\oplus x_0^{\sun}$, the vector $x_a^{\sun}$ is found as a translation of $x^{\sun}$; hence, 
$$
\overline{O(x_a^{\sun})}^{\sigma(X^{\sun},X)}\subset\overline{O(x^{\sun})}^{\sigma(X^{\sun},X)}.
$$
Consequently, if $\overline{O(x^{\sun})}^{\sigma(X^{\sun},X)}$ is separable, then $\overline{O(x_a^{\sun})}^{\sigma(X^{\sun},X)}$ is as well.
\end{remk}

\begin{remk} 
From Theorem \ref{apply-to-sun} (8), we learn that to obtain ergodic results for a semigroup, it makes sense to look for semigroups $\bk{S(t)}_{t\ge 0}$ such that $\bk{T(t)}_{t\ge} \subset \bk{S^{\sun}(t)}_{t\ge0}.$ With $\bk{T^{\sun\sun}(t)}_{t\ge 0}$, such a semigroup always exists. In particular, for the case of Eberlein weak almost periodicity using $jX\subset X^{\sun\sun}$ and for $x^{\sun} \in X^{\sun},$ we find
\begin{eqnarray*}
<T^{\sun\sun}(t)jx,x^*>&=&<jx,T^{\sun}(t)x^{\sun}>=<x,T^{\sun}(t)x^{\sun}>\\
&=& <T(t)x,x^{\sun}>=<jT(t)x,x^{\sun}>.
\end{eqnarray*}
Consequently, $ jT(t)x=T^{\sun\sun}(t)jx$, and because $X^{\sun}\subset X^*$, the weak compactness gives
$$
\overline{O(jx)}^{\sigma(X^{\sun\sun},X^{\sun})}=j\overline{O(x)}^{\sigma(X,X^{\sun})}\subset jX.
$$
Let $P^{\sun}$ denote the minimal idempotent with the semigroup $\bk{T^{\sun\sun}(t)}_{t\ge0};$ we find that
$$jX=(jX\cap R(P^{\sun}))\oplus (jX\cap N(P^{\sun})).$$ Hence, we may write $X=X_a\oplus X_0.$  Theorem \ref{apply-to-sun} (iv),(v) serves for $X_{ap}\subset X_a$, and for $x_0 \in X_0$, there exists a net such that $T(t_{\gamma})x_0\to 0$ weakly. Hence, Theorem \ref{RSdeco} becomes a consequence of Theorem \ref{apply-to-sun}, Theorem \ref{reccurent_equals_unimodular}, and the fact that $\overline{{T(t)x}_{t\ge0}}^{\sigma(X,X^*)}$ is separable.
\end{remk}

\begin{exa}
Let $p_{x^*}(f):=\sup_{t\in \rep}\btr{x^*(f(t)}$ for $x^*\in M$, with $X\subset M \subset X^{**},$ and $\tau:=\bk{p_{x^*}(\cdot):x^*\in M}.$
Then, $\sigma(X^*,M)$ is Hausdorff, and $\tau$ induces the space 
$$
WAAP_M(\rep,X^*):=\bk{f\in BUC(\rep,X^*): x^*(f)\in AAP(\rep), \mbox{ for all } x^*\in M}.
$$
Because for $f\in WAAP_M(\rep,X^*)$, we have that $<x,f(\cdot)>$ is weakly almost periodic, the semigroup, hence the $\scR_0^{\sun}$, is Abelian.
Additionally, if for $f\in WAAP_M(\rep,X^*)$, $\overline{ac}^{\tau}O(f)$ is $\tau$ compact and the set $\overline{O(f)}^{\tau}$ is norm separable (i.e., $\overline{O(f)}^{\tau}=\overline{O(f)}^{w^*}$), we obtain that the reversible part of $f$ is almost periodic whenever 
$$(L^1(\rep,X),BUC(\rep,X^*), \mbox{ translation semigroup})$$
 is a sun-dual triple. That $\sigma(BUC(rep,X^*),L^1(\rep,X))\subset \tau$ is a consequence of the uniform convergence and the denseness of the simple functions in $L^1(\re,X)$ with relatively compact support.
Moreover, it opens a machinery to several subspaces and topologies coming with $BUC(\rep,X^*),$ and the question regarding when it is a sun-dual with respect to the translation semigroup becomes quite obvious.
\end{exa}

\section{Regularity of the general splitting}

For Eberlein weakly almost periodic functions, a splitting with almost periodic functions is obtained. The following question naturally arises: Which additional periodicity property in the nonseparable case is given for $\bk{t\mapsto T^{\sun}(t)x^{\sun}}$ if $x\in X^{\sun}_a$? This will be answered next.
For given $x\in X$ and $x^{\sun}\in X^{\sun}$, we define
$$
\Funk{\fxx}{\rep}{\ce}{t}{<T^{\sun}(t)x^{\sun},x>.}
$$
Further, let $\bk{S(t)}_{t\ge 0}$ denote the translation semigroup.
Then,
for
$$
\Funk{S(t)}{L^1(\rep)}{L^1(\rep),}{h}{s\mapsto\left\{
\begin{array}{lcl}
0&:& s< t,\\
h(s-t)&:& s\ge t,
\end{array}
\right.
}
$$
clearly, $\bk{S(t)}_{t\ge 0}$ is a $C_0-$semigroup, i.e., 
$$
\Funk{S^{\sun}(t)}{BUC(\rep)}{BUC(\rep),}{f}{f_t.}
$$
Hence, we can define $BUC(\rep)_{rev}$ and obtain the following''
\begin{lem} \label{semigroup_to_translations}
\begin{enumerate}
\item If $\fxx\in BUC(\rep)_{a}$ for all $x\in X,$ then $x^{\sun}\in X^{\sun}_{rev}.$ 
\item If $x^{\sun}\in X_{rev}^{\sun}$, then $\fxx\in BUC(\rep)_{rev}.$
\end{enumerate} 
\end{lem}
\begin{proof}
To prove the first item, let $\fxx\in BUC(\rep)_{a}$ and $V=\netlim{\gamma}{\Gamma}T^{\sun}(t_{\gamma});$ we have to find a $W=\netlim{\al}{A}T^{\sun}(s_{\al})$ such that $WVx^{\sun}=x^{\sun}.$
Without loss of generality, $L=\netlim{\gamma}{\Gamma}S^{\sun}(t_{\gamma}),$ and the assumption $\fxx\in BUC(\rep)_{a}$ for all $x\in X$ leads to a net such that $U=\netlim{\al}{A}S^{\sun}(s_{\al})$ and $UL\fxx=\fxx$ for all $x\in X.$ Let $W=\netlim{\al}{A}T^{\sun}(s_{\al})$ and $h\in L^1(\rep);$, then
\begin{eqnarray*}
\int_{\rep}\fxx(t)h(t)dt&=&\netlim{\al}{A}\netlim{\gamma}{\Gamma}\int_{\rep}S^{\sun}(s_{\al})S^{\sun}(t_{\gamma})\fxx(t)h(t)dt\\
&=&\netlim{\al}{A}\netlim{\gamma}{\Gamma}\int_{\rep}<T^{\sun}(t)T^{\sun}(s_{\al})T^{\sun}(t_{\gamma})x^{\sun},x>h(t)dt\\
&=&\int_{\rep}<T^{\sun}(t)WVx^{\sun},x>h(t)dt \mbox{ for all } h\in L^1(\rep), \ x\in X.
\end{eqnarray*}
The continuity gives $\fxx(t)=<T^{\sun}(t)x^{\sun},x>=<T^{\sun}(t)WVx^{\sun},x>$ for all $t\ge 0$ and $x\in X.$
Consequently, $x^{\sun}=WVx^{\sun},$ which gives $x^{\sun}\in X^{\sun}_{rev}.$

To verify the second item, for given $L=\netlim{\al}{A}S^{\sun}(t_{\al})$, we have to find a $U=\netlim{\gamma}{\Gamma}S^{\sun}(s_{\gamma})$ such that $UL\fxx=\fxx.$
Without loss of generality, $V=\netlim{\al}{A}T^{\sun}(t_{\al})$ and $x^{\sun}\in X^{\sun}_{rev}$ lead to a net and an operator $W=\netlim{\gamma}{\Gamma}T^{\sun}(s_{\gamma}),$ with $WVx^{\sun}=x^{\sun}.$ We may assume that 
$U=\netlim{\gamma}{\Gamma}S^{\sun}(s_{\gamma}).$ With these settings, we have for $h\in L^1(\rep)$,
\begin{eqnarray*}
<UL\fxx,h>&=&\netlim{\al}{A}\netlim{\gamma}{\Gamma}\int_{\rep}S^{\sun}(s_{\gamma})S^{\sun}(t_{\al})\fxx(t)h(t)dt\\
&=&\netlim{\al}{A}\netlim{\gamma}{\Gamma}\int_{\rep}<T^{\sun}(t)T^{\sun}(s_{\gamma})T^{\sun}(t_{\al})x^{\sun},x>h(t)dt \\
&=&\int_{\rep}<T^{\sun}(t)x^{\sun},x>h(t)dt=<\fxx,h> \mbox{ for all } h\in L^1(\rep), \ x\in X.
\end{eqnarray*}
Hence, by continuity, $UL\fxx=\fxx.$
\end{proof}

\begin{defi}
Let 
$$
AP_{w^*}(\re,X^*):=\bk{f\in BUC(\re,X^*): \bk{t\mapsto <x,f(t)>}\in AP(\re) \mbox{ for all } x \in X}.
$$
\end{defi}
With the above, we obtain in the general case the following regularity for $x_a.$
\begin{theo} \label{weak-star-ap}
Let $\scS_0^{\sun}$ be Abelian; then, for all $x^{\sun}\in X^{\sun}_a$, there exists a 
$g \in AP_{w^*}(\re,X^*)$ such that
$$
T^{\sun}(t)x^{\sun}=g_{|\rep}.
$$
\end{theo}
\begin{proof}
If $\scS_0^{\sun}$ is Abelian and $X_{rev}^{\sun}=X_a^{\sun}$, then for given $x\in X, x^{\sun} \in X^{\sun},$ we have $\fxx\in W(\rep).$
Because $AP(\re)_{|\rep}=W(\rep)_{rev}=W(\rep)_a=BUC(\rep)_a\cap W(\rep),$ the above theorem becomes a consequence of Lemma \ref{semigroup_to_translations} (2).
\end{proof}

\section{Application I}
In this section, we will show how the previous observation is embedded into known theory, especially in the splitting of the special pairing  $(X,Y)=(L^1(\re,X),BUC(\re,X^*),$ where $\bk{T(t)}_{t\ge 0}$ denotes the translation semigroup. 
Moreover, we show how the compactification of the convex hull applies to obtain a mean.

Let $X$ be a Banach space; for $a\in\re$ and $\jz\in \bk{\re,\rep,[a,\infty)}$,
\begin{eqnarray*}
BUC(\jz,X)&:=&\bk{f:\jz\to X: f \mbox{ is bounded and uniformly continuous }}, \\
BUC_p(\jz,X)&:=& \bk{f\in BUC(\jz,X): f(\jz) \mbox{ is relatively compact }}.
\end{eqnarray*}

To verify that for the pairing $(X,Y),$  $Y=X^{\sun},$ we need the following property.
For the space above, we have the following.
\begin{defi}
A Banach space $X$ has the approximation property (a.p.) if for every compact $K\subset X$ and $\ep>0,$ there is a bounded finite-rank operator $T:X\to X$ such that $\nrm{Tx-x} \le \ep$ for all $x\in K.$
\end{defi}
The following theorem gives the sun-dual-pairing of $(X,Y).$
\begin{theo}[{\cite[p. 135, Theorem 7.3.11]{neervenLNM}}] \label{BUC-sun}
If $\bk{T_0(t)}_{t\in\re}$ is the translation semigroup on $L^{1}(\re),$ then for $\bk{T(t):=T_0(t)\otimes I}_{t\in\re}$, if $X^*$ has the a.p., we have
$$
L^1(\re,X)^{\sun}=BUC(\re,X^*).
$$
\end{theo}
From \cite[Prop. 2.1]{RuesErgdc}, we find that weakly almost periodic are uniformly continuous; hence, we obtain a splitting for $BUC(\re,X).$ 
\begin{cor}
If $X^*$ has the a.p. and $\bk{T_0(t)}_{t\in\re}$ is the translation semigroup on $L^{1}(\re),$ then for $\scS:=\bk{T(t):=T_0(t)\otimes I}_{t\in\re}$, dependent on the minimal idempotent $P^{\sun}\in \scS_0^{\sun},$ we find a splitting
$$
BUC(\re,X^*)=BUC(\re,X^*)_a\oplus BUC(\re,X^*)_0.
$$
This splitting is nontrivial because
$$
AP(\re,X^*)\subset BUC(\re,X^*)_a, \mbox{ and } \ W_0(\re,X^*)\subset BUC(\re,X^*)_0.
$$
\end{cor}
\begin{proof}
Because $BUC(\re,X^*)$ is a sun dual, we obtain by Theorem \ref{apply-to-sun} the splitting. Because 
$$AP(\re,X^*),W_0(\re,X)\subset BUC(\re,X),$$ 
we obtain $AP(\re,X^*)\subset BUC(\re,X^*)_a$ and $W_0(\re,X^*)\subset BUC(\re,X^*)_0.$
\end{proof}

We follow the definition of the minimal function from \cite[p. 908]{KnappDeco}, \cite[p. 346]{VeechProblems}. Very often, they coincide with recurrent \cite{BasitGuenzler_recurr} or reversible vectors  \cite[p. 105, Def. 4.3]{Krengel}. A very general theorem of equality is provided by \cite{Flohr}.
\begin{defi}
Let $f\in BUC(\re,X)$ and $\tau$ be a Hausdorff topology on $X;$ then, $f$ is called $\tau-$right minimal if, for every net $\net{t}{\la}{\Lambda}$, 
there exists a subnet $\bk{t_{\la_{\gamma}}}_{\gamma \in \Gamma}$ and a net $\net{s}{\al}{A}$ such that the limits with respect to $\tau$ fulfill, for some $g\in BUC(\re,X)$,
$$
 \netlim{\gamma}{\Gamma}  f(\cdot+t_{\la_{\gamma}})=g
$$
and 
$$
\netlim{\alpha}{A}g(\cdot + s_{\al})=f.
$$
\end{defi}

\begin{cor} \label{topolgy-equality}
Let $X$ be a Banach space. If $\tau$ is the compact open topology on $BUC_p(\re,X^*),$ then 
$$
\sigma(BUC_p(\re,X^*),L^1(\re,X))\subset \tau.
$$ 
If $\bk{T_0(t)}_{t\in\rep}$ is the translation semigroup on $L^1(\re,X)$, we have, for $BUC_p(\re,X^*)	\hookrightarrow L^1(\re,X)^{\sun},$ that
$$
BUC_p(\re,X^*)=BUC_p(\re,X)_a\oplus BUC_p(\re,X)_0.
$$
Moreover,
$$
(\overline{O(f)}^{\tau},\tau)=(\overline{O(f)}^{\tau},\sigma(BUC_p(\re,X^*),L^1(\re,X))). 
$$
Hence, the concept of $\tau-$right minimal and the definition given for $BUC_p(\re,X)_{rev}$ coincide.
\end{cor}
\begin{proof}
We simply have the embedding
$$
\Funk{i}{BUC_p(\re,X^*)}{L^1(\re,X)^{\sun}}{f}{\bk{g\mapsto \int_{\re}<f,g>d\mu}},
$$
which is $\nrm{\cdot}-\nrm{\cdot}$ continuous and $\tau-\sigma(L^1(\re,X)^{\sun},L^1(\re,X))$ continuous.
Note that because the vector-valued Arzela-Ascoli $(\overline{O(f)}^{\tau},\tau)$ is compact, we have 
$$
\overline{O(f)}^{\sigma(BUC_p(\re,X^*),L^1(\re,X))}=\overline{O(f)}^{\tau}\subset BUC_p(\re,X),
$$ 
and Theorem \ref{apply-to-sun} applies.

\end{proof}

\begin{theo} \label{BUC_a-description}
Let $X^*$ have the a.p., and let $\bk{T(t)}_{t\in\re}$ be the translation semigroup; then, 
$$
BUC(\re,X^*)_a \subset \bk{f: f \mbox{ is } \sigma(X^*,X)-\mbox{ right minimal }}.
$$
\end{theo}
\begin{proof} The definitions of $\kappa^{\sun}-$ minimal and $\sigma(X^*,X)-$right minimal coincide. By Theorem \ref{apply-to-sun}, we have $BUC(\re,X)_a\subset BUC(\re,X)_{rev}.$
\end{proof}

By \cite{KnappDeco}, we obtain the following theorem, which becomes a corollary to Theorem \ref{BUC_a-description} and Theorem \ref{apply-to-sun} in the case $S\in\bk{\re,\rep}.$
\begin{theo}\cite[Cor. 3.5]{KnappDeco}
Let $\jz\in\bk{\re,\rep}$ and $\bk{T(t)}_{t\ge 0}$ be the corresponding translation semigroup. Then,
\begin{enumerate}
\item $BUC(\jz)_{rev}\cap BUC(\jz)_{fl}=\bk{0}$,
\item $BUC(\jz)=BUC(\re)_{rev} + BUC(\jz)_{fl}$, and
\item $BUC(\jz)=A\oplus I$ in the notion of \cite{KnappDeco}.
\end{enumerate}
\end{theo}
\begin{proof}
From the relative compactness-open compactness, we obtain the first claim. The second claim is a consequence of Proposition \ref{recurr-results}, $BUC(\re)_a\subset BUC(\re)_{rev},$  $BUC(\re)_0\subset BUC(\re)_{fl}$ and $BUC(\re)=BUC(\re)_a\oplus BUC(\re)_0.$ To prove the last item, note that $P^{\sun}$ is a minimal idempotent in the sense of \cite[pp. 911-912, Thm. 3.4]{KnappDeco}; hence, $A=R(P^{\sun}),$ and $I=N(P^{\sun}).$
\end{proof}

\section{Application II}
Next, we show from the compactness of the convex semigroup $\scT_0^{\sun}$ that $N(A^{\sun})$ is complemented in $X^{\sun}.$

An application of $\scT_0^{\sun}$ is found in \cite{Gerlach}, where the theory of norming dual pairs is discussed. Note that $(X,X^{\sun},<\cdot,\cdot>)$ is such a norming dual pair. 
We recall that 
$$
C^{\sun}(r):=\frac{1}{r}\int_0^rT^{\sun}(s)ds \in \scT_0^{\sun}
$$
and
$$
(T^{\sun}(t)-I)C^{\sun}(r)x^{\sun}\to 0 \mbox{ in } \nrm{\cdot}.
$$
Thus, \cite[Lemma 4.5]{Gerlach} leads to the following.
\begin{cor} Let $\bk{T(t}_{t\ge 0}$ be a $C_0-$semigroup with generator $A$. Then, we have, for the mean of the dual semigroup and an appropriate net $\net{t}{\la}{\Lambda}$,
$$
\sigma(X^{\sun},X)-\netlim{\la}{\Lambda}C^{\sun}(r_{\la})x^{\sun} \in N(A^{\sun}),
$$
and
$
\kappa^{\sun}-\netlim{\la}{\Lambda}C^{\sun}(r_{\la})=Q^{\sun}$ is a projection onto $N(A^{\sun}).$ 
\end{cor}
\begin{proof}
By \cite[Lemma 4.5]{Gerlach}, we have $Q^{\sun}x^{\sun}\in N(A^{\sun}).$ Let $x^{ \sun}\in N(A^{\sun})$; then, 
$C(r)x^{\sun}\equiv x^{\sun}=Q^{\sun}x^{\sun}.$ It remains to prove that $Q^{\sun}Q^{\sun}=Q^{\sun}.$ 
If $x^{\sun}\in X^{\sun}$ and $Q^{\sun}x^{\sun}=y^{\sun}\in N(A^{\sun})$, then 
$$
Q^{\sun}Q^{\sun}x^{\sun}=Q^{\sun}y^{\sun}=y^{\sun}=Q^{\sun}x^{\sun},
$$
which concludes the proof.
\end{proof}

\section{Examples}
In this section, we present some counterexamples. We start with an example of a $\sigma(BUC(\re),L^1(\re))-$flight vector, which fails to be a $\sigma(BUC(\re),BUC(\re)^*)-$flight vector.
Throughout this section, let 
$$
\Funk{f}{\re}{\re,}{t}{\sin(\ln(\btr{t}+1)).}
$$
This function is taken from \cite{RuessInt}.
We recall the following obvious result from functional analysis.
\begin{pro} \label{interchange}
Let $\net{x}{\la}{\Lambda}\subset X$ and $ \net{x^*}{\gamma}{\Gamma}\subset X^*,$ with 
$\sigma(X,X^*)-\netlim{\la}{\Lambda} x_{\la}=x$ and $\sigma(X^*,X)-\netlim{\gamma}{\Gamma}x^*_{\gamma}=x^*;$ then,
$$
\netlim{\gamma}{\Gamma}\netlim{\lambda}{\Lambda}<x^*_{\gamma},x_{\la}>=\netlim{\lambda}{\Lambda}\netlim{\gamma}{\Gamma}<x^*_{\gamma},x_{\la}>=<x^*,x>.
$$
\end{pro}
Next, we show the following:
\begin{pro} \label{0-separated}
$$ 0\not\in \overline{\bk{f_t:t\in\re}}^{\sigma(BUC(\re),BUC(\re)^*)}. $$
\end{pro}
\begin{proof}
Assume that $0\in \overline{\bk{f_t:t\in\re}}^{\sigma(BUC(\re),BUC(\re)^*)};$ then, there is a net $\net{s}{\gamma}{\Gamma}\subset \re$ such that $\sigma(BUC(\re,),BUC(\re)^*)-\netlim{\gamma}{\Gamma}f_{s_{\gamma}}= 0,$ and for all 
$(t_m)_{m\in\za},$ $\netlim{\la}{\Lambda}\delta_{t_{m_{\la}}}=x^*,$ we have
$$
\netlim{\la}{\Lambda}\netlim{\gamma}{\Gamma}f(s_{\gamma}+t_{m_{\la}})=\netlim{\la}{\Lambda}\netlim{\gamma}{\Gamma}<f_{s_{\gamma}},\delta_{t_{m_{\la}}}>=0.
$$
By Proposition \ref{interchange}, we have
$$
\netlim{\gamma}{\Gamma}\netlim{\la}{\Lambda}f(s_{\gamma}+t_{m_{\la}})=0.
$$
However, for $t_m=\exp(2m\pi+\pi/2),$ i.e., $t_{m_{\la}}=\exp(2m_{\la}\pi+\pi/2),$ we find that, for $t_{m_{\la}}>s_{\gamma},$
\begin{eqnarray*}
\netlim{\la}{\Lambda}\sin(\ln(t_{m_{\la}}+s_{\gamma}+1))&=&\netlim{\la}{\Lambda}\sin\fk{\ln\fk{\exp(2m_{\la}\pi+\pi/2)\fk{1+\frac{s_{\gamma}+1}{\exp(2m_{\la}\pi+\pi/2)}}}} \\
&=&\netlim{\la}{\Lambda}\sin\fk{2m_{\la}\pi+\pi/2 +\ln\fk{1+\frac{s_{\gamma}+1}{\exp(2m_{\la}\pi+\pi/2)}}} \\
&=&\netlim{\la}{\Lambda}\sin\fk{\pi/2 +\ln\fk{1+\frac{s_{\gamma}+1}{\exp(2m_{\la}\pi+\pi/2)}}} \\
&=& 1,
\end{eqnarray*}
which is a contradiction.
\end{proof}
\begin{cor} \label{separated}
\begin{enumerate}
\item Let $x^*=\netlim{\la}{\Lambda}\delta_{t_{m_{\la}}}$ with $t_{m_{\la}}=\exp(2m_{\la}\pi+\pi/2);$ then,
$$
x^*_{|\overline{\bk{f_t:t\in\re}}^{\sigma(BUC(\re),BUC(\re)^*)}}\equiv 1.
$$
Hence, 
$$ \overline{\bk{f_t:t\in\re}}^{\sigma(BUC(\re)^{**},BUC(\re)^*)} \subset (x^*)^{-1}(\bk{1})\not\in 0.$$
Moreover, let $g(t)\equiv \nu <1;$ then,
$$
\nu=<x^*,g> < \ \alpha \le \ <x^*,f_t> \mbox{ for all } t\in\re
$$
for some $\nu <\alpha < 1.$ \\
\item Choosing $\tau\in \sin^{-1}\bk{\nu},$ $\nu\in [-1,1]$, we obtain for $t_m=\exp(2m\pi+\tau)$ that for the subnet $\net{m}{\la}{\Lambda}$, 
$$
\netlim{\lambda}{\Lambda}\delta_{t_{m_{\la}}}=:x^*_{\nu},
$$ and
$$
<x^*_{\nu},f_s>=\nu.
$$
Moreover, let $g(t)\equiv 1;$ then,
$$
1=<x^*_{\nu},g> \  > \ \alpha \ge \ <x^*_{\nu},f_t> \mbox{ for all } t\in\re
$$
for some $\alpha < 1.$
\end{enumerate}
Thus, 
\begin{equation}
[-1,1]\cap \overline{\bk{f_t:t\in\re}}^{\sigma(BUC(\re)^{**},BUC(\re)^*)}=\emptyset,
\end{equation}
and
\begin{equation} \label{w-star-closure}
\overline{\bk{f_t:t\in\re}}^{\sigma(BUC(\re)^{**},BUC(\re)^*)} \subset \bigcap_{\nu\in [-1,1]} x_{\nu}^{-1}(\bk{\nu}).
\end{equation}

\end{cor}
By a similar construction, we obtain for similar translations and the pointwise topology that
$$\scT:=\bk{\delta_t:t\in\re}.$$
\begin{remk} \label{pointwise-Orbit}
By choosing $\tau\in \sin^{-1}\bk{\nu},$ $\nu\in [-1,1]$, we find that, for $s_n=\exp(2k_n\pi+\tau)$ and $f_{s_n},$

$$
\bk{f_t}_{t\in\re}\cup[-1,1]\subset \overline{\bk{f_t:t\in\re}}^{\scT}.
$$
Because $f$ is even, it is sufficient to consider a net $\net{s}{\la}{\Lambda}\subset \rep,$ $s_{\la}=\exp(2k_{\la}\pi+\tau_{ \la}).$ We may assume that $\tau_{\la}\to \tau\in[0,2\pi],$ and we obtain
\begin{equation} 
\bk{f_t}_{t\in\re}\cup [-1,1]\supset \overline{\bk{f_t:t\in\re}}^{\scT}.
\end{equation}
\end{remk}

\begin{proof} From the proof of Proposition \ref{0-separated}, we learn that
\begin{eqnarray} \label{possible-limits}
\netlim{\la}{\Lambda}\sin(\ln(t+s_{\la}+1))&=&\netlim{\la}{\Lambda}\sin\fk{\tau_{\la} +\ln\fk{1+\frac{t+1}{\exp(2k_{\la}\pi+\tau_{\la})}}},
\end{eqnarray}
which concludes the proof.
\end{proof}
The same remark will hold for $\scTco,$ and we obtain the following.

\begin{remk} \label{compact-open-Orbit}
By choosing $\tau\in \sin^{-1}\bk{\nu},$ $\nu\in [-1,1]$, we find that, for $s_n=\exp(2k_n\pi+\tau)$,

$$
\bk{f_t}_{t\in\re}\cup[-1,1]\subset \overline{\bk{f_t:t\in\re}}^{\scTco}.
$$
Because $\scTco$ is metric, for a given sequence $s_n=\exp(2k_n\pi+\tau_n)$, we may assume that $\tau_n\to \tau\in[0,2\pi],$ and we obtain
$$
\bk{f_t}_{t\in\re}\cup [-1,1]\supset \overline{\bk{f_t:t\in\re}}^{\scTco}.
$$
\end{remk}
From the above observation, it is clear that $f$ is not Eberlein weakly almost periodic. 
\begin{cor}
$$
\overline{O(f)}^{\scTco}\not=\scS_0f,
$$
where $\scS_0$ is the compactification of the bounded operators of translations $\bk{T(t)}_{t\in\re}$ coming from \cite{Witz}.
\end{cor}
\begin{proof}
We showed in Remark \ref{compact-open-Orbit} that $0\in \overline{O(f)}^{\scTco},$ but by Corollary \ref{separated}, $0$ is weak$^*$ separated from the $O(f)$ in $BUC(\re)^{**};$ hence, $0$  is not in the weak$^*$ closure of the orbit with respect to $S_0.$.
\end{proof}
Because the pointwise topology is weaker than the weak topology $\scT\subset \sigma(BUC(\re),BUC(\re,X)^*),$ an application of Corollary \ref{separated} leads to the following corollary:
\begin{cor}
$$
\overline{\bk{f_t:t\in\re}}^{\sigma(BUC(\re),BUC(\re)^*)}=\bk{f_t:t\in\re}.
$$ 
\end{cor}

Therefore, it remains for us to compute the weak* closure. The question is what is $\omega(f)$ considered in the w* topology in $BUC(\re)^{**}?$ The previous study showed that the pointwise topology will not provide a hint. The missing weak compactness may serve for elements in $BUC(\re)^{**}\setminus BUC(\re).$

The pointwise solution is not an approach. Moreover, we have the following:
\begin{remk}
$$
\fk{\overline{\bk{f_t:t\in\re}}^{\sigma(BUC(\re)^{**},BUC(\re)^*)},\scT} \mbox{ is not Hausdorff. }
$$
\end{remk}
\begin{proof}
Because $\net{f}{t}{\re}$ is bounded, there is a subnet $\net{s}{\la}{\Lambda}$ such that 
$$
g=\sigma(BUC(\re)^{**},BUC(\re)^*)-\netlim{\la}{\Lambda}f_{s_{\la}}.
$$ 
An application of \ref{possible-limits} leads to a $\nu\in [-1,1]$ such that 
$$
f_{s_{\la}} \to \nu
$$
pointwise.
However, in light (\ref{w-star-closure}) of Corollary \ref{separated}, we have
$$
\nu\not\in \overline{\bk{f_t:t\in\re}}^{\sigma(BUC(\re)^{**},BUC(\re)^*)}.
$$
Hence, the topology $\scT$ cannot separate the $\nu$ from the weak$^*$ closure of the orbit. Thus, in the pointwise topology,
$$
g-\nu =0,
$$
but $g$ and $\nu$ are separated in the w$^*$ topology.
\end{proof}

\begin{remk}
$f$ is uniformly continuous, but by \cite{RuessInt}, the function fails to be Eberlein weakly almost periodic. Therefore, in view of Theorem \ref{S0-Abelian}, with respect to the translation semigroup, the pairing $(L^1(\re), BUC(\re))$ forms a non-Abelian pairing, while $(C_0(\re),L^1(\re))$ from an Abelian one.
\end{remk}

\begin{exa} \label{weak-null-but-not-Eberlein}
Let $t_m:=16^m$ and $s_n:=16^{n+1};$ by \cite[Example 3.1]{RRS}, we have, for
$$
E=\bk{t_n\pm t_m:m\le n},
$$
that ${\chi_{E\cup -E}}_{|\za}\in W(\za)$.
Further, let
$$
\Funk{\vp}{\rep}{\re,}{t}{\left\{ 
\begin{array}{rcr} 
4(\frac{1}{4}-s)&:& 0\le s\le\frac{1}{4}, \\ 
0 &:& s>\frac{1}{4}, 
\end{array} \right.}
$$
and for $M\subset \za$,
$$
\Funk{g_M}{\re}{\re,}{t}{\sum_{k\in M}\vp(\btr{t-k})}
$$
if $F:=(E\cup-E)\cap\za.$ We have \cite[Remark 3.4]{RRS} that $g:=g_F \in W_0(\rep)$ if ${\rm lb} $ is the binary logarithm 
$$
\Funk{f}{\re}{\re,}{t}{\sin(\frac{\pi}{8}\mbox{${\rm lb}$}(\btr{t}+1)),}
$$
and $h(t):=g(t)f(t).$
Then,
$$
h(s_n+t_m)=g(s_n+t_m)f(s_n+t_m)=f(s_n+t_m) \ \ \forall \ m\le n,
$$
and for some null sequences $\seq{\beta}{n},\seq{\al}{n}$, we have
\begin{eqnarray*}
f(s_n+t_m)&=&\sin(\frac{\pi}{8}\lb(\btr{s_n+t_m}+1))\\
&=&\sin(\frac{\pi}{8}\mbox{\lb}(16^{n+1}(1+\al_n)))=\sin(\frac{4n\pi}{8}+\frac{\pi}{2}+\beta_n).
\end{eqnarray*}
Hence, for the subsequence $n=4k,$ $\ilm{n}f(s_{4n}+t_m)=1$,
$\ilm{m}h(s_n+t_{4m})=\ilm{m}f(s_n+t_{4m})=0;$ consequently, $h\not\in W(\rep),$ by \cite[Double limit criterion]{Groth}, but for some 
$\seq{\om}{k}$, we have
$$
g_{\omega_n}\to 0 \mbox{ weakly in } BUC(\rep).
$$
Therefore, for all $\seq{t}{m}\subset \rep$, the double limits of 
$h(\om_n+t_m)$ are $0;$ hence,
$$
h_{\om_n}\to 0 \mbox{ weakly in } BUC(\rep)
$$
again by Grothendieck's double limit criterion \cite{Groth}.
Summarizing, we find a function for which some translations converge weakly to $0.$   Therefore,
$h\in BUC(\rep)_0$ by Theorem \ref{apply-to-sun} but fails to be Eberlein weakly almost periodic. Moreover,
if $E:=\bk{P^{\sun}\in E(\scS^{\sun}_0): \ \le_L\mbox{-minimal}}$, then
$$
h \in \bigcap_{P^{\sun}\in E }N(P^{\sun});
$$
thus, the intersection of all $N(P^{\sun})$ does not reduce to the Eberlein weakly almost periodic functions.
\end{exa}
\begin{exa} \label{Abelian_but_not_Eberlein}
Let $h_1=0$ and, for $n\ge 2,$ $\funk{h_n}{\re}{[0,1]}$ equicontinuous, $h_n(2^{2n+1})=h_n(2^{2(n+1)})=0$, 
${h_n}_{|[2^{2n+1}+1,2^{2(n+1)}-1]}\equiv 1,$  and $h_n$ equal 0 otherwise. By definition, we have $\supp{h_n}\cap \supp{h_m}=\emptyset$ for $n\not=m.$ With these functions, we define
$$
\Funk{g}{\re}{l^2(\za),}{t}{\left\{
\begin{array}{ccl}
h_n(t)e_n&:& t\in [2^{2n+1},2^{2(n+1)}], \\
0&:& \mbox{otherwise}.
\end{array} \right.
}
$$
If $Y=\overline{span}\bk{g_{\tau}:\tau\in\re},$ then $Y\subset BUC(\re,l^2(\za)),$
and ${\scS_0^{\sun}}_{|Y}=\bk{T(t)}_{t\in\re}\cup \bk{0},$ where 
$\bk{T(t)}_{t\in\re}$ denotes the translation group. Then, ${\scS_0^{\sun}}_{|Y}$ is Abelian, but 
$\bk{t\mapsto T(t)g}$ fails to be Eberlein weakly almost periodic. Note that for $f\in Y,$ $\overline{O(f)}^{\sigma(Y,L^1(\re,l^2))}=\scS_0^{\sun}f\subset Y.$
\end{exa}
\begin{proof}
To compute the $w^*OT$ closure of $\bk{T(t)}_{t\ge }$ restricted to $Y,$ we have to consider for $h\in L^1(\re,l^2(\za))$,
\begin{eqnarray*}
\btr{\int_{\re}<g_t,h>d\mu} &\le&\int_{-R}^R\sum_{n=2}^{\infty}h_n(s+t)\btr{<e_n,h(s)>}ds +\ep\\
&\le& \int_{-R}^R\btr{<e_n,h(s)>}ds +\ep,
\end{eqnarray*}
when for large $t,$ $t+s\in t+[-R,R]\cap [2^{2n+1},2^{2(n+1)}]\not=\emptyset;$ hence, we are in the situation of Lebesgue's dominated convergence theorem, and we obtain the limit of 0. For $t\to -\infty,$ we choose $t<-R$, and we find that the integral is equal to $0.$
Consequently, a finite linear combination of translations of $g$ converges to zero. The proof for the uniform limit is straightforward. To verify $g$ not being Eberlein weakly almost periodic, apply \cite[Theorem 2.1]{RuessInt} with $\om_n=2^{2n}$ and $(t_m,x_m^*)=(2^{2m+1}+1,e_m).$ 
\end{proof}

\end{document}